\documentclass[11pt]{article}%
\usepackage{amssymb,amsmath,amsfonts,amsthm,array,bm,bbm,color}%
\usepackage{mathrsfs}
\usepackage{graphicx}
\providecommand{\U}[1]{\protect\rule{.1in}{.1in}}

\usepackage[
a4paper,
left=22mm,
right=22mm,
top=22mm,
bottom=25mm
]{geometry}

\numberwithin{equation}{section}

\newtheorem{theorem}{Theorem}[section]
\newtheorem{lemma}[theorem]{Lemma}
\newtheorem{corollary}[theorem]{Corollary}

\newtheorem{remark}[theorem]{Remark}
\newtheorem{example}[theorem]{Example}
\newtheorem{definition}[theorem]{Definition}

\newtheorem{assumption}[theorem]{Assumption}

\def\<{\langle}
\def\>{\rangle}
\def\E{\mathbb{E}}
\def\P{\mathbb{P}}
\def\R{\mathbb{R}}
\def\T{\mathbb{T}}
\def\Z{\mathbb{Z}}

\begin{document}
	\title{Structure preservation and emergent dissipation in stochastic wave equations with transport noise}
	\author{Chang Liu$^{2,3}$\footnote{Email: liuchang2021@amss.ac.cn}
		\quad Dejun Luo$^{1,2}$\footnote{Email: luodj@amss.ac.cn} \medskip \\
		{\footnotesize $^1$SKLMS, Academy of Mathematics and Systems Science, Chinese Academy of Sciences, Beijing 100190, China}\\
		{\footnotesize $^2$School of Mathematical Sciences, University of Chinese Academy of Sciences, Beijing 100049, China}\\
		{\footnotesize $^3$Academy of Mathematics and Systems Science, Chinese Academy of Sciences, Beijing 100190, China}}
	\maketitle
	
	\vspace{-20pt}
	
	\begin{abstract}
		We study nonlinear wave equations perturbed by transport noise acting either on the displacement or on the velocity. Such noise  models random advection and, under suitable scaling of space covariance, may generate an effective dissipative term. We establish well-posedness in both cases and analyse the associated scaling limits. When the noise acts on the displacement, the system preserves its original structure and converges to the deterministic nonlinear wave equation, whereas if it acts on the velocity, the rescaled dynamics produce an additional Laplacian damping term, leading to a stochastic derivation of a Westervelt-type acoustic model.
	\end{abstract}
	

	\section{Introduction}\label{sec-intro}
	
	Waves arise throughout nature and in a wide range of scientific applications, including the vibration of a string, the propagation of sound in air, and the motion of water on a lake. They are typically described by hyperbolic partial differential equations. Over the years, many nonlinear wave models have been introduced to capture phenomena in optics, elasticity, fluid dynamics, and high intensity acoustics.
	In many realistic situations, however, the medium exhibits rapidly fluctuating structures or irregular motion induced by turbulence or environmental noise, which are difficult to be fully represented by deterministic models. To account for these effects, stochastic perturbations are introduced, leading naturally to stochastic wave equations (SWEs), which were originally studied in the 1960s to model wave propagation in randomly fluctuating media.
	
	We briefly recall some previous works among the vast literature on SWEs. The early papers in the proceedings \cite{Chow81} deal with wave propagation and the effect of multiple scattering in random media. Later, Carmona and Nualart \cite{CarNua88a, CarNua88b} studied the smoothness of solutions and propagation of singularities for random nonlinear wave equations. Mueller \cite{Mueller1997} investigated the long time existence for SWEs with multiplicative noise in low spatial dimensions $d\leq 2$, see \cite{DalFran98} for related results in dimension 2. For SWEs with polynomial nonlinearity, Chow \cite{Chow 2002} analysed how random perturbations affect blow-up phenomena and global existence in a Sobolev space setting. In the recent paper \cite{MilSanz21}, Millet and Sanz-Sol\'e established global existence results for SWEs with superlinear coefficients. We refer to Dalang's review article \cite{Dalang09} for more results on the properties of solutions to SWEs.
	
	In this work, we consider the SWE on the $d$-dimensional torus $\T^d$ for $d\geq 2$:
	\begin{equation}\label{eq-intro SWE}
		\partial_t^2 u=\Delta u+f(u)+\dot{W} \circ \nabla \Phi(u,\partial_t u),
	\end{equation}
	where $f(u)$ is the nonlinear term, assumed to be globally Lipschitz in $u$, the circle $\circ$ means Stratonovich multiplication, and $W=W(t,x)$ is the time-space noise of the form
	\begin{equation}\label{eq-noise}
		W(t,x)= \sum_{k\geq 1} \sigma_k(x) B_t^k.
	\end{equation}
	Here $\{B_t^k\}_{k}$ are independent Brownian motions and $\{\sigma_k \}_k$ are divergence-free spatial vector fields describing the space structure of noise. Finally, $\Phi(u,\partial_t u)=u$ or $\Phi(u,\partial_t u)=\partial_t u$ distinguishes whether the noise acts on the displacement or on the velocity.
	
	Equation \eqref{eq-intro SWE} generalizes the classical wave equation by including stochastic transport-type perturbations $\dot{W} \circ \nabla u$ or $\dot{W} \circ \nabla \partial_t u$, which model random advection and have been extensively studied in the context of fluid dynamics and transport equations. Early works \cite{BCF92, MR04, MR05} investigated stochastic Navier-Stokes equations with transport-type noise, while subsequent studies \cite{BFM16, FGP10, LangCri23} further explored its effects on linear transport equations or 2D Euler equations, particularly in the Stratonovich formulation, whose rigorous derivations were provided in \cite{FP21, FP22, DDHolm 15}. A remarkable discovery, initiated in Galeati \cite{Gal20}, is that under an appropriate rescaling of the spatial covariance, small-scale transport noise may generate an effective dissipative term in the limiting dynamics. This mechanism has since been shown to regularize a variety of models \cite{CL23, FGL21 JEE, FLL24, Luo21, PFH23, QS24}; see also \cite{FlanLuongo23} for an overview. In particular, such dissipative effects can suppress potential blow-up in nonlinear dynamics: for reaction–diffusion systems, multiplicative transport noise can ensure global existence or arbitrarily delay blow-up \cite{Agresti 24}, while for 3D Navier–Stokes equations on the torus, stochastic transport noise or deterministic transport-type perturbations can regularize the dynamics and yield global well-posedness even for large initial data \cite{Agr24b, FHLN22, FL21 PTRF}. These results demonstrate that small-scale stochastic fluctuations may have a significant impact on large-scale behaviour.
	
	Motivated by these developments in transport noise and classical nonlinear wave theory, we investigate how transport-type stochastic perturbations influence nonlinear wave propagation. It is worth mentioning that the large-scale behaviour depends sharply on whether the noise perturbs the displacement $u$ or the velocity $\partial_t u$. As we shall see, when the noise perturbs the displacement $u$, the structure of the wave equation is essentially preserved; in particular, the It\^o-Stratonovich correction vanishes, and no dissipative term appears.  By contrast, when the same family of noises acts on the velocity $\partial_t u$,  the It\^o-Stratonovich correction in the scaling limit of noises induces an additional term $\kappa \Delta \partial_t u$, which plays the role of a Laplacian damping acting on the time derivative. This structure closely resembles the dissipative operator appearing in Westervelt-type equations.
	
	The classical Westervelt equation, derived by Westervelt \cite{Westervelt1963} in the context of nonlinear acoustics, incorporates a quadratic nonlinearity while neglecting viscous and thermal losses:
	$$\partial_t^2 u=\mu_1 \Delta u+\mu_2 \partial_t^2(u^2).$$
	Beyond its mathematical interest, this class of nonlinear acoustic models has been widely employed in medical and industrial applications, such as diagnostic ultrasound, lithotripsy, ultrasound cleaning and sonochemistry; see, for example, \cite{AC1999,DKBR2000,FF IEEE,SVE IEEE}.
	Subsequent developments extended Westervelt's original formulation by systematically including dissipative mechanisms, such as viscosity and thermal conduction, thereby introducing a third-order temporal derivative term that accounts for acoustic diffusivity, see e.g. \cite{HB2024, NJ2016, SK2015}. In particular, Shevchenko and Kaltenbacher  \cite{SK2015} derived the generalized Westervelt equation from fundamental fluid dynamics, which has an additional term $\Delta \partial_t u$ to model acoustic damping:
	$$\partial_t^2 u=\mu_1 \Delta u+\mu_2 \partial_t^2(u^2)+\mu_3 \Delta \partial_t u.$$
	Formally letting $\mu_2 \rightarrow 0$ in the above equation recovers precisely the limiting equation obtained in our stochastic scaling analysis with the nonlinearity removed (see Theorem \ref{thm-scaling limit SWE2} below). This connection highlights an interesting fact: the damping term typically introduced from physical mechanisms such as viscosity or thermal conduction may also arise from stochastic perturbations of the wave dynamics.
	
	Before describing our main results, we introduce the notation used throughout the paper. For $s \in \mathbb{R}$, $H^s=H^s(\mathbb{T}^d)$ represents the inhomogeneous Sobolev space equipped with the norm
	$$	\|u\|_{H^s} = \big\|(I-\Delta)^{s/2} u\big\|_{L^2}= \Big( \sum_{j \in \mathbb{Z}^d} (1+|j|^2)^s |\hat u(j)|^2 \Big)^{1/2},$$
	where $\hat u(j)$ stands for the Fourier coefficient of $u$. We denote by $H^{-s}$ the dual space of $H^s$ and write $H^0=L^2$. In addition, the notation $\< \cdot, \cdot \>$ stands for the inner product in $L^2$, or the duality between elements in $H^s$ and $H^{-s}$. For a sequence $\theta=\{\theta_j \}_{j\in \Z^d}$ and $p\in [1,\infty]$, we denote $\|\theta\|_{\ell^p}$ the usual $\ell^p$-norm of $\theta$. We also write $a\lesssim b$ if $a\leq Cb$ for some constant $C>0$, and $a\lesssim_\gamma b$ when the dependence of $C$ on a parameter $\gamma$ needs to be emphasized. For simplicity, we will write $\sum_{k}$ for $\sum_{k \geq 1}$ in the sequel.

	\subsection{Well-posedness}\label{subs-well-posedness}
	In order to study the well-posedness results for two types of SWEs, we begin with specifying the assumptions imposed on the noises and the nonlinear term.
	\begin{assumption}\label{asp-1}
		(A1) We assume that the vector fields $\{\sigma_k\}_k$ in \eqref{eq-noise} are divergence-free and define the covariance function of noise as
		$$Q(x,y)=\E \big[W(1,x)\otimes W(1,y)\big]=\sum_{k} \sigma_{k}(x)\otimes \sigma_{k}(y).$$
		We assume in addition that the noise is space homogeneous and isotropic, namely
		\begin{equation}\label{eq-asp on Q}
			Q(x,y)=Q(x-y); \quad Q(x,x)=Q(0)=2\kappa I_d,
		\end{equation}
		where $I_d$ is the $d$-dimensional identity matrix and $\kappa>0$ is a constant.
		
		(A2) The nonlinear term $f(u)$ is assumed to be globally Lipschitz in $u$.
	\end{assumption}
	
	With these assumptions at hand, we first consider the Cauchy problem associated with \eqref{eq-intro SWE} in the case $\Phi(u,\partial_t u)=u$, namely,
	$$\partial_t^2 u=\Delta u+f(u)+\dot{W} \circ \nabla u, \quad \big(u(0),\partial_t u(0)\big)=(u_0, v_0).$$
	Letting $v=\partial_t u$ and using the expression of noise \eqref{eq-noise}, we obtain the equation for velocity variable:
	$$dv=\Delta u \, dt+f(u) \, dt+\sum_{k} \sigma_k \cdot \nabla u \circ dB_t^k.$$
	Since $du=v\, dt$, it is easy to see that the It\^o-Stratonovich correction vanishes:
	$$d \, \big[\sigma_k \cdot \nabla u, B^k\big]_t=0;$$
	now the equation \eqref{eq-intro SWE} with $\Phi(u,\partial_t u)=u$  reduces to the following It\^o form:
	\begin{equation}\label{eq-SWE}\tag{SWE 1}
		\left\{
		\begin{aligned}
			&\partial_t u=v,  \\
			& d v=\Delta u \, dt+f(u) \, dt+\sum_{k} \sigma_k \cdot \nabla u \, dB_t^k,\\
			&\big(u(0),v(0) \big)=(u_0,v_0).
		\end{aligned}
		\right.
	\end{equation}
	Below we will discuss this It\^o equation rather than the Stratonovich one, as the It\^o form is more convenient for our analysis and allows us to apply energy estimates and  It\^o formula directly.

	The following theorem establishes the pathwise well-posedness of system \eqref{eq-SWE}.
	\begin{theorem}\label{thm-SWE1 wellposedness}
		Given any $(u_0,v_0) \in H^1 \times L^2$ and $T>0$, equation \eqref{eq-SWE} admits a pathwise unique weak solution $(u,v)$ on the interval $[0,T]$.
		
		More precisely, given a filtered probability space $(\Omega,\mathcal{F},(\mathcal{F}_t),\P)$ and a family of independent $(\mathcal{F}_t)$-Brownian motions $\{B^k\}_k$ defined on $\Omega$, there exists a pathwise unique $(\mathcal{F}_t)$-progressively measurable process $(u,v)\in L^2(\Omega,L^\infty([0,T];H^1 \times L^2))$, with continuous trajectories in $C([0,T];L^2 \times H^{-1})$, $\P$-a.s., such that  for any test function $\phi\in C^\infty(\T^d)$, it holds $\P$-a.s. for all $t\in [0,T]$,
		\begin{equation}\label{eq-def of weak sol}
			\begin{split}
				\<u(t),\phi\>&=\< u_0, \phi\>+\int_{0}^{t} \<v(s),\phi\> \, ds, \\
				\<v(t),\phi\>&=\<v_0,\phi\>+\int_{0}^{t} \<u(s), \Delta \phi \> \, ds+\int_{0}^{t} \big\<f(u(s)),\phi\big\> \, ds-\sum_{k} \int_{0}^{t} \big\<u(s),\sigma_k \cdot \nabla \phi \big\> \, dB^k_s.
			\end{split}
		\end{equation}
		Furthermore, the weak solution $(u,v)$ satisfies the following uniform bound:
		$$\E \Big[ \|u\|_{L^\infty([0,T];H^1)}^2+\|v\|_{L^\infty([0,T];L^2)}^2\Big] \leq C_T \big(\|u_0\|_{H^1}^2+\|v_0\|_{L^2}^2+1\big).$$
	\end{theorem}
	We remark that the solution obtained above is weak in the analytical sense, whereas it is probabilistically strong and can therefore be defined on any given probability space.
	
	Having established the well-posedness of the first model, we now turn to the Cauchy problem for a different type of stochastic perturbation, in which the noise acts on $\partial_t u$:
	$$\partial_t^2 u=\Delta u+f(u)+\dot{W} \circ \nabla \partial_t u, \quad \big(u(0),\partial_t u(0)\big)=(u_0, v_0).$$
	Actually, we can also let $\partial_t u=v$ and get the first-order Stratonovich equation for velocity variable $v$; then we utilize the assumption \eqref{eq-asp on Q} to transform it into the It\^o form with an additional Laplacian term. Now it reads as
	\begin{equation}\label{eq-SWE2}\tag{SWE 2}
		\left\{
		\begin{aligned}
			&\partial_t u=v,  \\
			& d v=\Delta u \, dt+f(u) \, dt+\sum_{k} \sigma_k \cdot \nabla v \, dB_t^k+\kappa \Delta v \, dt,\\
			&\big(u(0),v(0) \big)=(u_0,v_0).
		\end{aligned}
		\right.
	\end{equation}
	For system \eqref{eq-SWE2}, we can establish the existence of weak solutions whose precise meaning is given as follows.
	
	\begin{definition}\label{def-SWE2 weak solution}
		Let $(u_0,v_0) \in H^1 \times L^2$ and $T>0$. We say that \eqref{eq-SWE2} has a weak solution if there exist a filtered probability space $(\Omega,\mathcal{F},(\mathcal{F}_t),\P)$, a family of independent $(\mathcal{F}_t)$-Brownian motions $\{B^k\}_k$ and a pair of stochastic processes $(u,v)$ defined on $\Omega$, such that
		\begin{itemize}
			\item $(u,v)$ is $(\mathcal{F}_t)$-progressively measurable and belongs to
			$ L^2\big(\Omega;L^\infty([0,T];H^1 \times L^2)\big)$;
			\item for any test function $\phi\in C^\infty(\T^d)$ and all $t\in [0,T]$, the pair $(u,v)$ satisfies the following weak formulation $\P$-a.s.,
			\begin{equation}\label{eq-SWE2 sol def}
				\begin{split}
					\<u(t),\phi\>&=\< u_0, \phi\>+\int_{0}^{t} \<v(s),\phi\> \, ds, \\
					\<v(t),\phi\>&=\<v_0,\phi\>+\int_{0}^{t} \<u(s), \Delta \phi \> \, ds+\int_{0}^{t} \big\<f(u(s)),\phi\big\> \, ds+\kappa \int_{0}^{t} \<v(s), \Delta \phi \> \, ds\\
					&\quad-\sum_{k} \int_{0}^{t} \big\<v(s),\sigma_k \cdot \nabla \phi \big\> \, dB^k_s.
				\end{split}
			\end{equation}
		\end{itemize}
	\end{definition}
	With this definition at hand, we can now state a weak existence result for \eqref{eq-SWE2}.
	\begin{theorem}\label{thm-SWE2 wellposedness}
		For any initial value $(u_0,v_0) \in H^1 \times L^2$ and $T>0$, there exists a  weak solution $(u,v)$ on the time interval $[0,T]$ in the sense of Definition \ref{def-SWE2 weak solution}. Moreover, the solution satisfies the uniform bound $\P$-a.s.,
		$$\|u\|_{L^\infty([0,T];H^1)} \vee \|v\|_{L^\infty([0,T];L^2)} \leq C_{ub},$$
		where $C_{ub}>0$ is a finite constant depending on $T$, $\|u_0\|_{H^1}$ and $\|v_0\|_{L^2}$.
	\end{theorem}
	
	We briefly discuss the main ideas of proofs for the two theorems above and explain how our approach differs from the standard strategies in the SPDE literature.
	\begin{remark}\label{rmk-wellposedness}
		A standard approach in the well-posedness theory of SPDEs is to first construct a probabilistically weak solution via compactness arguments, and then establish pathwise uniqueness, thereby promoting it to a probabilistically strong solution.
		
		In the case of the first system \eqref{eq-SWE}, we adopt an alternative way: we directly construct a sequence of approximate solutions that is Cauchy in a suitable space, without relying on compactness methods. This yields a probabilistically strong solution which is also pathwise unique.
		By contrast, for the second model \eqref{eq-SWE2}, our analysis follows the classical compactness framework. While this allows us to construct a probabilistically weak solution, we are currently unable to establish pathwise uniqueness in this setting. As a consequence, the solution obtained for \eqref{eq-SWE2} remains probabilistically weak.
	\end{remark}

	\subsection{Scaling limits}
	
	We now turn to the scaling limit of these two models, focusing on the behaviour of the solutions as the noise becomes asymptotically uncorrelated at different space points. Throughout the analysis of scaling limit, we impose the following assumption.
	\begin{assumption}\label{asp-2}
		For every $N\ge 1$, let $\{\sigma^N_k\}_{k\geq 1}$ be a family of divergence-free vector fields and  $Q^N$ the associated covariance function. We assume that for each $N\geq 1$, $Q^N$ satisfies (A1) in Assumption \ref{asp-1} with a constant $\kappa>0$ independent of $N$. We also require that
		\begin{equation}\label{eq-asp on theta}
			\lim_{N\rightarrow\infty} \big\|Q^N\big\|_{L^1} =0.
		\end{equation}
	\end{assumption}
	
	Then for each $N\geq 1$, we consider equation \eqref{eq-intro SWE} driven by a sequence of noises
	$$W^N(t,x)=\sum_{k} \sigma_k^N(x) B_t^k.$$
	To emphasize the dependence on $N$, we denote the corresponding solutions by $u^N$, and set $v^N=\partial_t u^N$. As in Section \ref{subs-well-posedness}, we distinguish two cases: $\Phi(u^N,\partial_t u^N)=u^N$ and $\Phi(u^N,\partial_t u^N)=\partial_t u^N$.
	For simplicity, we always take $(u^N(0),v^N(0))\equiv (u_0,v_0)$ for all $N\geq 1$.
	
	We start with the first system, i.e.,
	\begin{equation}\label{eq-scaling limit}
		\left\{
		\begin{aligned}
			&\partial_t u^N=v^N,  \\
			& d v^N=\Delta u^N \, dt+ f(u^N) \, dt+\sum_{k}  \sigma_k^N \cdot \nabla u^N \, dB_t^k.
		\end{aligned}
		\right.
	\end{equation}
	By Theorem \ref{thm-SWE1 wellposedness}, for any $N\geq 1$, system \eqref{eq-scaling limit} has a pathwise unique solution $\{u^N_t\}_{t\in [0,T]}$ with uniform bound
	$$\E \, \|u^N\|_{L^\infty([0,T];H^1)}^2 \leq C_T \big(\|u_0\|_{H^1}^2+\|v_0\|_{L^2}^2+1\big).$$
	Under Assumption \ref{asp-2}, we have the following scaling limit result:
	\begin{theorem}\label{thm-scaling limit SWE1}
		Given $(u_0,v_0) \in H^1\times L^2$ and time $T>0$, the solutions $\{u^N\}_{N\geq 1}$ to equations \eqref{eq-scaling limit}  converge to the unique weak solution $\bar{u}$ of the following deterministic wave equation as $N\rightarrow \infty$:
		\begin{equation}\label{eq-limit wave equation}
			\partial_t^2 \bar{u}=\Delta \bar{u}+f(\bar{u}), \quad \big(\bar{u}(0),\partial_t \bar{u}(0) \big)=(u_0,v_0).
		\end{equation}
		Furthermore,  the following strong convergence result holds for any $\gamma \in (0,\frac{1}{2})$:
		$$\lim_{ N\rightarrow \infty} \E \, \big\|u^N-\bar{u}\big\|_{C([0,T];H^{1-\gamma})}^2=0.$$
	\end{theorem}
	Here we only present a qualitative convergence result without giving a convergence rate. In fact, one may try to express the solution $u^N$ of \eqref{eq-scaling limit} and the solution $\bar{u}$ to limit equation \eqref{eq-limit wave equation} in mild form through the Green function of the deterministic wave operator, whose explicit form depends on the spatial dimension. Then quantitative convergence rate might be established in some suitable space by following the approach in \cite{FLD quantitative}. Since this analysis involves additional technicalities beyond the main scope of the present work, here we omit it.

	Next, we turn to the second system, in which the stochastic perturbation acts on the velocity $\partial_t u^N$ and an additional term $\kappa \Delta v^N$ is present in the It\^o form:
	\begin{equation}\label{eq-scaling limit SWE2}
		\left\{
		\begin{aligned}
			&\partial_t u^N=v^N,  \\
			& d v^N=\Delta u^N \, dt+ f(u^N) \, dt+\kappa \Delta v^N \, dt+\sum_k  \sigma_k^N \cdot \nabla v^N \, dB_t^k.
		\end{aligned}
		\right.
	\end{equation}
	The presence of this extra Laplacian term leads to its appearance in the limiting equation as well. Recall from Theorem \ref{thm-SWE2 wellposedness} that the sequence $\{(u^N,v^N)\}_{N}$ satisfies the uniform bound
	\begin{equation}\label{eq-asp on v ub}
		\P\text{-a.s.,}\quad \sup_{N\geq 1} \Big(\|u^N\|_{L^\infty([0,T];H^1)} \vee \|v^N\|_{L^\infty([0,T];L^2)} \Big)\leq C_{ub}.
	\end{equation}
	Under Assumption \ref{asp-2}, we now state the corresponding scaling limit result.
	
	\begin{theorem}\label{thm-scaling limit SWE2}
		Suppose $(u_0,v_0) \in H^1\times L^2$, then the solutions $\{u^N\}_{N\geq 1}$ of equations \eqref{eq-scaling limit SWE2}  converge, as $N\rightarrow \infty$, to the unique weak solution $\bar{u}$ of the following equation:
		$$\partial_t^2 \bar{u}=\Delta \bar{u}+f(\bar{u})+\kappa \Delta \partial_t \bar{u}, \quad \big(\bar{u}(0),\partial_t \bar{u}(0) \big)=(u_0,v_0).$$
		If we further assume $f\in C_b^2(\R)$, then for $a \in (0,\frac{1}{2})$ and $\epsilon \in (0,a]$, the following quantitative estimate holds for $d=2,3$:
		$$\E \, \|u^N-\bar{u}\|_{C([0,T];H^{-a}(\T^d))} \lesssim_{T,\kappa} C_{ub} \big\|Q^N\big\|_{L^1}^\frac{a-\epsilon}{d}.$$
	\end{theorem}
	
	The above estimate provides a bound in $C([0,T];H^{-a})$ for $a \in (0, \tfrac{1}{2})$, which is sufficiently small by the assumption \eqref{eq-asp on theta}. Since the sequence $\{u^N\}_N$ is uniformly bounded in $L^\infty([0,T];H^1)$ $\P$-a.s. by \eqref{eq-asp on v ub}, and a similar bound holds for $\bar{u}$ by analogous arguments. It then follows from interpolation that the estimate can be upgraded to the space $C([0,T];H^{1-})$. We note that our quantitative result is stated for the cases $d=2,3$; see Remark \ref{rmk-dimension restrction} for an explanation.
	
	\begin{remark}\label{rmk-scaling limit}
		The limit equation obtained above can also be written as a first-order system
		\begin{equation}\label{eq-limit WE2}
			\left\{
			\begin{aligned}
				&\partial_t \bar{u}=\bar{v},  \\
				& \partial_t \bar{v}=\Delta \bar{u}+ f(\bar{u}) +\kappa \Delta \bar{v}.
			\end{aligned}
			\right.
		\end{equation}
		The presence of the Laplacian term $\Delta \bar{v}$ in the second equation, and correspondingly the term $\Delta v^N$ in the second equation of \eqref{eq-scaling limit SWE2}, allow us to express both $\bar{v}$ and $v^N$ in mild form via the associated heat kernel. Then we can exploit standard analytic properties of the heat kernel (such as Lemmas \ref{lemma-semigroup prop 2} and \ref{lemma-semigroup property} below) to derive quantitative estimates.
	\end{remark}

	\subsection{Organization of the paper}
	The remainder of this paper is organized as follows. In Section \ref{sec-preliminaries}, we give an example of noise satisfying Assumption \ref{asp-1}-(A1)  and derive the corresponding estimates for the covariance function. We also collect several analytical preliminaries, including estimates for the stochastic convolution. Sections \ref{sec-wellposedness for SWE1} and \ref{sec-prf of SWE2 wellposed} are devoted to the proofs of Theorems \ref{thm-SWE1 wellposedness} and \ref{thm-SWE2 wellposedness}, respectively.	Finally, the scaling limit results for the systems \eqref{eq-scaling limit} and \eqref{eq-scaling limit SWE2} are proved in Sections \ref{sec-scaling limit SWE1} and \ref{sec-scaling limit SWE2}.

	\section{Preliminaries}\label{sec-preliminaries}
	We first present an example of divergence-free vector fields in any dimension $d\geq 2$.
	\begin{example}\label{eg-noise}
		The divergence-free vector fields can be taken as
		$$\sigma_{k,i}(x)= \theta_k a_{k,i} e_{k}(x), \quad x\in \T^d, \,  k\in \Z_0^d, \, i=1,2,\ldots, d-1.$$
		Here $\Z_0^d=\Z^d \backslash \{0\}$ is the non-zero lattice points and $\theta=\{\theta_k\}_k \in \ell^2(\Z_0^d)$ is  a square summable sequence such that $\sum_{k} \theta_k^2=\frac{d}{d-1}\kappa$ for some $\kappa>0$. We decompose the lattice as $\Z_0^d=\Z_+^d \cup \Z_-^d$ satisfying $\Z_+^d=-\Z_-^d$. For each $k \in \Z_+^d$, we choose $\{a_{k,i}\}_{i=1}^{d-1}$ to be an orthonormal basis of $k^\perp=\{y\in \R^d: y\cdot k=0\}$ and define $a_{k,i}=a_{-k,i}$ for $k \in \Z_-^d$.  Finally, $e_k(x)=e^{2\pi ik \cdot x}$ for $x\in \T^d$, $k\in \Z_0^d$.
		
		In two dimensions, we can define $k^\perp=(k_2,-k_1)$ and choose explicitly $a_{k,1}=a_k=\frac{k^\perp}{|k|}$ for $k\in \Z_+^2$ and then set $a_k=a_{-k}$ for $k\in\Z_-^2$.
		
		It is easy to see that, for any $d\geq 2$, the covariance function is spatially homogeneous since
		\begin{equation}\label{eq-covariance function}
			Q(x,y)=Q(x-y)=\sum_{k\in \Z_0^d} \theta_k^2  \Big(I_d-\frac{k \otimes k}{|k|^2}\Big) e^{2\pi i k\cdot(x-y)}.
		\end{equation}
		Following the arguments developed for the three-dimensional case in \cite[Section~2]{FL21 PTRF},  one can obtain $Q(0)=2\kappa I_d$ for any $d\geq 2$. Hence (A1) in Assumption \ref{asp-1} is reasonable.
	\end{example}
	
	\begin{remark}
		For the covariance function in Example \ref{eg-noise}, a direct estimate of $\|Q\|_{L^1}$ seems difficult. However, thanks to the explicit formula \eqref{eq-covariance function}, one can first control the $L^2$-norm of $Q$:
		$$\|Q \|_{L^2}^2 =(d-1)\sum_{k} \theta_k^4 \lesssim_d \|\theta \|_{\ell^\infty}^2 \sum_{k} \theta_k^2 \leq C \|\theta \|_{\ell^\infty}^2.$$
		And then it follows immediately on the torus that
		$$\|Q \|_{L^1} \leq \|Q \|_{L^2} \lesssim \|\theta \|_{\ell^\infty}.$$
		On the other hand, the representation \eqref{eq-covariance function} also implies a direct bound in Fourier space: $$\|\widehat{Q}\|_{\ell^\infty} \leq \|\theta\|_{\ell^\infty}^2.$$ 
		Consequently, if we take a sequence of coefficients $\{\theta^N\}_N$ in Example \ref{eg-noise} such that $\|\theta^N\|_{\ell^\infty} \rightarrow 0$ as $N\rightarrow \infty$, then both $\big\|Q^N\big\|_{L^1}$ and $\big\|\widehat{Q^N}\big\|_{\ell^\infty}$ vanish in the limit. Moreover, the Fourier bound exhibits a faster decay rate than the $L^1$-estimate. 
		
		Nevertheless, for the sake of simplicity, we adopt the assumption \eqref{eq-asp on theta} in terms of the $L^1$-norm of $Q$, which appears more natural.
	\end{remark}

	We then introduce two well known properties of the heat semigroup $\{e^{\kappa t \Delta}\}_{t\geq 0}$ for $\kappa>0$, which will be used in the proof of quantitative bound in Theorem \ref{thm-scaling limit SWE2}.
	\begin{lemma}\label{lemma-semigroup prop 2}
		Let $g\in H^\tau$, $\tau \in \R$. For any $\rho \geq 0$, it holds
		$$\big\|e^{\kappa t\Delta} g\big\|_{H^{\tau+\rho}} \lesssim (\kappa t)^{-\rho/2} \|g\|_{H^\tau}.$$
	\end{lemma}
	
	\begin{lemma}\label{lemma-semigroup property}
		For any $s_1<s_2$, $\tau \in \R$ and $g \in L^2([s_1,s_2];H^\tau)$, we have
		$$\int_{s_1}^{s_2} \bigg\|\int_{s_1}^{t} e^{\kappa (t-r)\Delta} g_r \, dr \bigg\|^2_{H^{\tau+2}}  dt \lesssim \kappa^{-2} \int_{s_1}^{s_2} \|g_r\|_{H^\tau}^2 \, dr.$$
	\end{lemma}
	
	Applying the two lemmas above, we obtain two estimates for the stochastic convolution
	$$Z_t=\sum_{k} \int_{0}^{t} e^{\kappa (t-s)\Delta} \big(\sigma_{k} \cdot \nabla v_s\big) \, dB^k_s,$$
	which will be used in Section \ref{subsec-quantitative SWE2}. The vector fields $\{\sigma_k\}_k$ have covariance function $Q(x,y)$ satisfying \eqref{eq-asp on Q} and $v$ satisfies the uniform bound given in Theorem \ref{thm-SWE2 wellposedness}:
	$$\|v\|_{L^\infty([0,T];L^2)} \leq C_{ub}, \quad \P \text{-a.s.}.$$
	\begin{lemma}\label{lemma-stochastic convolution}
		Suppose $\kappa>0$ and $v$ satisfies the above bound. For any $\epsilon \in (0,1/2)$, one has
		$$\E \, \bigg[\sup_{t\in [0,T]} \big\|Z_t\big\|_{H^{-\epsilon}}^2\bigg]\lesssim_{\epsilon,T} \kappa^{\epsilon}  C_{ub}^2;$$
		similarly, the following estimate holds in a larger space:
		$$\E \, \bigg[\sup_{t\in [0,T]} \big\|Z_t\big\|_{H^{-\frac{d}{2}-\epsilon}}^2\bigg] \lesssim_{\epsilon,T} \kappa^{\epsilon-1} C_{ub}^2 \|Q\|_{L^1}.$$
	\end{lemma}
	\begin{proof}
		Since the proof closely follows the argument in \cite[Lemma 2.5]{FLD quantitative}, we only outline the distinct steps here.
		By Burkholder-Davis-Gundy inequality and Lemma \ref{lemma-semigroup prop 2}, it holds
		$$\aligned
		\E \, \big\|Z_t\big\|_{H^{-\epsilon}}^2 &\lesssim  \E \bigg[\sum_{k} \int_{0}^{t} \big\|e^{\kappa (t-s)\Delta} \big(\sigma_{k} \cdot \nabla v_s\big)\big\|_{H^{-\epsilon}}^2 \, ds\bigg]\\
		&\leq \E \bigg[\sum_{k} \int_{0}^{t} \frac{1}{\kappa^{1-\epsilon} (t-s)^{1-\epsilon}} \big\|\sigma_{k} \cdot \nabla v_s\big\|_{H^{-1}}^2 \, ds \bigg].
		\endaligned $$
		Noticing that $\{\sigma_k\}_k$ are divergence-free, we can apply \eqref{eq-asp on Q} to further obtain
		$$\E \, \big\|Z_t\big\|_{H^{-\epsilon}}^2 \leq \E \bigg[\sum_{k} \int_{0}^{t} \frac{1}{\kappa^{1-\epsilon}(t-s)^{1-\epsilon}} \big\|\sigma_{k} v_s\big\|_{L^{2}}^2 \, ds \bigg]\lesssim \kappa^{\epsilon}  \, \E \int_{0}^{t} \frac{1}{(t-s)^{1-\epsilon}} \|v_s\|_{L^2}^2 \,ds.$$
		By the $\P$-a.s. bound $\|v\|_{L^\infty([0,T];L^2)} \leq C_{ub}$, we estimate the integral and arrive at
		$$\E \, \big\|Z_t\big\|_{H^{-\epsilon}}^2 \lesssim t^\epsilon \kappa^{\epsilon} \epsilon^{-1} C_{ub}^2.$$
		The remaining arguments for the first estimate are similar to those in \cite[Lemma 2.5]{FLD quantitative}, and we omit them for brevity.
		
		For the second result, we also apply Burkholder-Davis-Gundy inequality and Lemma \ref{lemma-semigroup prop 2} to get
		$$\aligned
		\E \, \big\|Z_t\big\|_{H^{-\frac{d}{2}-2\epsilon}}^2&\lesssim  \E \bigg[\sum_{k} \int_{0}^{t} \big\|e^{\kappa (t-s)\Delta} \big(\sigma_{k} \cdot \nabla v_s\big)\big\|_{H^{-\frac{d}{2}-2\epsilon}}^2 \, ds\bigg]\\
		&\leq \E \bigg[\sum_{k} \int_{0}^{t} \frac{1}{\kappa^{1-\epsilon} (t-s)^{1-\epsilon}} \big\|\sigma_{k} \cdot \nabla v_s\big\|_{H^{-1-\frac{d}{2}-\epsilon}}^2 \, ds \bigg].
		\endaligned $$
		Applying the definition of Sobolev norm, it holds
		$$\sum_{k} \big\|\sigma_{k} \cdot \nabla v_s\big\|_{H^{-1-\frac{d}{2}-\epsilon}}^2=\sum_{k} \sum_{l \in \Z^d} \big(1+|l|^2\big)^{-1-\frac{d}{2}-\epsilon}\big| \< \sigma_k \cdot \nabla v_s, e_l \>\big|^2.$$
		Using the divergence-free property of $\{\sigma_k\}_k$, we have
		$|\<\sigma_k\cdot\nabla v_s,e_l\>|=|\<v_s,\sigma_k\cdot\nabla e_l\>|$, and thus
		$$\sum_{k} \big| \< \sigma_k \cdot \nabla v_s, e_l \>\big|^2=\big\< v_s\nabla e_l, Q\ast (v_s \nabla e_l) \big\> \leq \|v_s \nabla e_l\|_{L^2} \|Q \ast (v_s \nabla e_l)\|_{L^2} \leq  \|v_s \nabla e_l\|_{L^2}^2 \|Q\|_{L^1},$$
	    where the last step follows from Young's inequality for convolution.
	    Substituting this estimate into the previous equality yields
		$$\sum_{k} \big\|\sigma_{k} \cdot \nabla v_s\big\|_{H^{-1-\frac{d}{2}-\epsilon}}^2 \leq \sum_{l \in \Z^d} \big(1+|l|^2\big)^{-1-\frac{d}{2}-\epsilon} |l|^2 \|Q\|_{L^1} \|v_s\|_{L^2}^2  \lesssim \|Q\|_{L^1} \|v_s\|_{L^2}^2.$$
		Indeed, the last inequality follows from the fact that $\sum_{l\in\Z^d} (1+|l|^2)^{-\frac{d}{2}-\epsilon}<\infty$.
		Hence we can utilize the uniform bound of $v$ to obtain
		$$\E \, \big\|Z_t\big\|_{H^{-\frac{d}{2}-2\epsilon}}^2\lesssim \|Q\|_{L^1}  \E \bigg[\int_{0}^{t} \kappa^{\epsilon-1} |t-s|^{\epsilon-1} \|v_s\|_{L^2}^2 \, ds\bigg] \leq t^\epsilon \kappa^{\epsilon-1} \epsilon^{-1} C_{ub}^2 \|Q\|_{L^1} .$$
		We omit the remaining arguments hereafter, cf. \cite[Lemma 2.5]{FLD quantitative}.
	\end{proof}
	
	Thanks to Lemma \ref{lemma-stochastic convolution}, we can directly deduce by interpolation that
	\begin{corollary}\label{coro-stochastic convolution}
		For $a\in (0,1)$ and $\epsilon \in (0,a]$, it holds
		$$\E \bigg[\sup_{t\in [0,T]}\|Z_t\|_{H^{-a}}^2  \bigg]  \lesssim_{\epsilon,T} \kappa^{\epsilon-\frac{2(a-\epsilon)}{d}} C_{ub}^2 \|Q\|_{L^1}^\frac{2(a-\epsilon)}{d}.$$
	\end{corollary}
	
	In the end, we present the following compact embedding theorem from \cite[Corollary 5]{JSimon}.
	\begin{theorem}\label{thm-compact embedding}
		Suppose $\rho>\frac{d}{2}+1$ and $\gamma \in (0,\frac{1}{2})$, then these two embeddings are compact:
		\begin{itemize}
			\item  $L^\infty([0,T];H^1) \cap W^{1,\infty} ([0,T];L^2) \hookrightarrow C([0,T];H^{1-\gamma});$
			\item  $L^\infty([0,T];L^2) \cap W^{\frac{1}{3},4}([0,T];H^{-\rho}) \hookrightarrow C([0,T];H^{-\gamma}).$
		\end{itemize}
	\end{theorem}

	\section{Well-posedness for \eqref{eq-SWE}} \label{sec-wellposedness for SWE1}
	In this section, we prove Theorem \ref{thm-SWE1 wellposedness} by the method of Galerkin approximation. We always assume the initial value $(u_0,v_0) \in H^1 \times L^2$ unless otherwise specified. Define finite dimensional projection $\Pi_n: L^2 \rightarrow H_n=\text{span}\{e_{l}: |l|\leq n\}$, where $\{e_l\}_{l\in \Z^d}$ is the orthonormal Fourier basis of $L^2(\T^d)$.
	Then we study the approximate equation:
	\begin{equation}\label{eq-finite dimension}
		\left\{
		\begin{aligned}
			&\partial_t u_n=v_n,  \\
			& d v_n=\Delta u_n \, dt+\Pi_n f(u_n) \, dt+\sum_k \Pi_n (\sigma_k \cdot \nabla u_n) \, dB_t^k,
		\end{aligned}
		\right.
	\end{equation}
	with initial data $(u_n(0),v_n(0))=(\Pi_n u_0,\Pi_n v_0) \in H^1\times L^2$. We aim to show the solutions $\{(u_n,v_n)\}_n$ is a Cauchy sequence in some Banach space, thereby yielding the existence of a limiting process $(u,v)$, which in turn solves \eqref{eq-SWE}.
	
	To begin with, we define the energy associated with each approximation in \eqref{eq-finite dimension}:
	\begin{equation}\label{eq-def of energy}
		\mathcal{E}(u_n(t)):=\|u_n(t)\|_{H^1}^2+\|v_n(t)\|_{L^2}^2, \quad \forall n\geq 1.
	\end{equation}
	Similarly,  the initial energy for \eqref{eq-SWE} reads as
	$$\mathcal{E}(u(0))=\|u_0\|_{H^1}^2+\|v_0\|_{L^2}^2.$$
	The following lemma provides a uniform bound on the energy.
	\begin{lemma}\label{lemma-energy estimate}
		For any $T>0$, there exists $C_T>0$, such that the uniform estimate holds:
		$$\sup_{n\geq 1}	\E \Big[\sup_{s\leq T} \mathcal{E}(u_n(s)) \Big] \leq C_T \big(\mathcal{E}(u(0))+1\big).$$
	\end{lemma}
	\begin{proof}
		In the proof, we fix an arbitrary $n\geq 1$ and estimate $d \mathcal{E}(u_n(t))$ first. Notice that
		$$d\|u_n\|_{H^1}^2=d \big\|(I-\Delta)^\frac{1}{2} u_n\big\|_{L^2}^2=2\big\<(I-\Delta)^\frac{1}{2} u_n,(I-\Delta)^\frac{1}{2} v_n\big\> \, dt=2\big\<(I-\Delta) u_n, v_n\big\> \,dt;$$
		besides, by It\^o formula and the self-adjoint property of $\Pi_n$, it holds
		$$\aligned
		&\quad d\|v_n\|_{L^2}^2=2\<v_n, dv_n\>+\<d v_n, dv_n\>\\
		&=2\<v_n, \Delta u_n\> \, dt+2\big\<v_n, f(u_n)\big\>\, dt+ 2\sum_k\big\<v_n,\sigma_k \cdot \nabla u_n \big\> \, dB_t^k+\sum_k \big\|\Pi_n \big(\sigma_k \cdot \nabla u_n\big) \big\|_{L^2}^2 \, dt.
		\endaligned $$
		Then we apply the definition of $\mathcal{E}(u_n(t))$ to arrive at
		$$\aligned
		d \mathcal{E}(u_n(t))&=2\<u_n,v_n\> \, dt+2\<v_n,f(u_n)\>\, dt+ 2\sum_k\<v_n,  \sigma_k \cdot \nabla u_n \> \, dB_t^k+\sum_k \big\|\Pi_n \big(\sigma_k \cdot \nabla u_n\big) \big\|_{L^2}^2 \, dt \\
		&\leq 2\|v_n\|_{L^2} \big(\|u_n\|_{L^2} + \|f(u_n)\|_{L^2}\big) \, dt+2\sum_k\big\<v_n,  \sigma_k \cdot \nabla u_n \big\> \, dB_t^k+2\kappa \, \|\nabla u_n\|_{L^2}^2 \, dt,
		\endaligned $$
		where in the last step we used the following estimate under the assumption \eqref{eq-asp on Q}:
		\begin{equation}\label{eq-Q L2}
			\sum_k \big\|\Pi_n (\sigma_k \cdot \nabla u_n) \big\|_{L^2}^2 \leq \sum_k \|\sigma_k \cdot \nabla u_n\|_{L^2}^2= \int_{\mathbb{T}^d} \big(\nabla u_n(x)\big)^\top Q(x,x) \nabla u_n(x) \, dx= 2\kappa \|\nabla u_n\|_{L^2}^2.
		\end{equation}
		Since $f$ is Lipschitz, it holds $|f(u)|\leq C(1+|u|)$, which leads to $\|f(u_n)\|_{L^2}^2 \leq C\big(1+\|u_n\|_{L^2}^2\big)$. Thus we can apply Young's inequality for the above formula to obtain
		$$\aligned
		d \mathcal{E}(u_n(t)) &\leq C\big(\|v_n\|_{L^2}^2+\|u_n\|_{L^2}^2+1\big)\, dt+2\sum_k\<v_n,  \sigma_k \cdot \nabla u_n \> \, dB_t^k+2\kappa \, \|\nabla u_n\|_{L^2}^2 \, dt\\
		&\leq C_\kappa \big(\mathcal{E}(u_n(t))+1\big) \, dt+2\sum_k\<v_n,  \sigma_k \cdot \nabla u_n \> \, dB_t^k.
		\endaligned $$
		
		To estimate the stochastic term, we fix $R>0$ and then define a stopping time
		$$\tau_R^n:=\inf \big\{t\geq 0: \mathcal{E}(u_n(t))\geq R\big\}, \quad \forall n\geq 1.$$
		Recall that $\mathcal{E}(u_n(0)) \leq \mathcal{E}(u(0))$; integrating the above inequality with respect to time and taking supremum yield
		$$\sup_{s\leq t} \mathcal{E}\big(u_n(s\wedge \tau_R^n)\big) \leq \mathcal{E}(u(0))+C_\kappa \int_{0}^{t\wedge \tau_R^n}  \big(\mathcal{E}(u_n(r))+1\big) \, dr+2\sup_{s\leq t\wedge \tau_R^n} \Big|\sum_k \int_{0}^s \<v_n,  \sigma_k \cdot \nabla u_n \> \, dB_r^k\Big|.$$
		Taking expectation and applying Burkholder-Davis-Gundy inequality, we arrive at
		$$\aligned
		\E \Big[\sup_{s\leq t} \mathcal{E}\big(u_n(s\wedge \tau_R^n)\big) \Big] &\lesssim \mathcal{E}(u(0))+\E  \int_{0}^{t} \big(\mathcal{E}(u_n(s\wedge \tau_R^n))+1\big) \, ds\\
		&\quad+  \E \bigg[\bigg(\int_{0}^{t\wedge \tau_R^n} \sum_k \<v_n, \sigma_k \cdot \nabla u_n\>^2 \,ds\bigg)^{\frac{1}{2}}\bigg].
		\endaligned $$
		Now we estimate the final term by Cauchy inequality and similar computations in \eqref{eq-Q L2}:
		$$\aligned
		\E \bigg[\bigg(\int_{0}^{t\wedge \tau_R^n} \sum_k \<v_n, \sigma_k \cdot \nabla u_n\>^2 \,ds\bigg)^{\frac{1}{2}}\bigg] &\lesssim_\kappa \E \bigg[\bigg(\int_{0}^{t\wedge \tau_R^n} \|v_n\|_{L^2}^2 \|\nabla u_n\|_{L^2}^2 \, ds\bigg)^{\frac{1}{2}}\bigg]\\
		&\leq \E \bigg[\bigg(\sup_{s\leq t} \|v_n(s\wedge \tau_R^n)\|_{L^2}^2 \int_{0}^{t\wedge \tau_R^n} \|\nabla u_n(s)\|_{L^2}^2 \, ds\bigg)^{\frac{1}{2}}\bigg]\\
		&\leq \varepsilon \, \E \Big[\sup_{s\leq t} \mathcal{E}\big(u_n(s\wedge \tau_R^n)\big) \Big]+C(\varepsilon) \, \E \bigg[\int_{0}^{t\wedge \tau_R^n} \mathcal{E}(u_n(s)) \, ds\bigg],
		\endaligned $$
		where in the last step we used Young's inequality with $\varepsilon>0$ a constant and $C(\varepsilon)$ depending on $\varepsilon$.	Inserting this estimate into the above one, we choose $\varepsilon$ sufficiently small to get
		$$\E \Big[\sup_{s\leq t} \mathcal{E}\big(u_n(s\wedge \tau_R^n)\big) \Big] \leq C_{\kappa,\varepsilon} \bigg( \mathcal{E}(u(0))+1+\int_{0}^{t} \E \Big[\sup_{r\leq s} \mathcal{E}\big(u_n(r\wedge \tau_R^n)\big) \Big] \, ds\bigg).$$
		By Gr\"onwall's lemma, we arrive at
		\begin{equation}\label{eq-sup estimate}
			\E \Big[\sup_{s\leq t} \mathcal{E}\big(u_n(s\wedge \tau_R^n)\big) \Big] \leq C_{\kappa,\varepsilon,T} \big(\mathcal{E}(u(0))+1\big).
		\end{equation}
		
		If we expect to take limit $R \rightarrow \infty$ such that the final result is independent of stopping time, we need to show that, for any $n\geq 1$, $\P$-a.s., $\tau_R^n \rightarrow \infty $ as $R\rightarrow \infty$. We discuss by contradiction. Suppose there exist $\delta>0$ and finite time $T_0>0$, such that for any $R \gg 1$,
		$$\P(\tau_R^n<T_0) >\delta.$$
		Then for $t\geq T_0$, it holds
		$$\E \Big[\sup_{s\leq t} \mathcal{E}\big(u_n(s\wedge \tau_R^n)\big) \Big] \geq \E \Big[\textbf{1}_{\{\tau_R^n <T_0\}}\sup_{s\leq t} \mathcal{E}\big(u_n(s\wedge \tau_R^n)\big) \Big] \geq R \, \P(\tau_R^n<T_0 ) >R\delta,$$
		which is impossible since the left-hand side is finite by \eqref{eq-sup estimate}, but the right-hand side tends to $\infty$ as $R\rightarrow \infty$. Thus by Fatou's lemma, we can take limit $R\rightarrow \infty$ in \eqref{eq-sup estimate} and finally arrive at
		$$	\E \Big[\sup_{s\leq T} \mathcal{E}(u_n(s)) \Big] \leq C_{\kappa,\varepsilon,T} \big(\mathcal{E}(u(0))+1\big).$$
		Noticing that the right-hand side is independent of $n$, the proof is complete.
	\end{proof}
	
	Below we will show that $\{(u_n,v_n)\}_{n}$  is a Cauchy sequence in $L^2\big(\Omega;C([0,T]; L^2\times H^{-1})\big)$.
	For $m> n\geq 1$, we denote
	$$w_{m,n}=u_m-u_n, \quad z_{m,n}=v_m-v_n,$$
	then equation \eqref{eq-finite dimension} yields
	\begin{equation}\label{eq-difference}
		\left\{
		\begin{aligned}
			&\partial_t w_{m,n}=z_{m,n},  \\
			& d z_{m,n}=\Delta w_{m,n} \, dt+\big(\Pi_m f(u_m)-\Pi_n f(u_n) \big)\, dt+\sum_k \big(\Pi_m (\sigma_k \cdot \nabla u_m) -\Pi_n (\sigma_k \cdot \nabla u_n) \big)\, dB_t^k.
		\end{aligned}
		\right.
	\end{equation}
	To analyse the difference $\{(w_{m,n},z_{m,n})\}_{m>n}$, we introduce the energy
	$$\mathcal{N}_{m,n}(t):=\|w_{m,n}(t)\|_{L^2}^2+\|z_{m,n}(t)\|_{H^{-1}}^2.$$
	The next lemma shows that this energy vanishes as $m,n \rightarrow \infty$.
	
	\begin{lemma}\label{lemma-Cauchy sequence for SWE1}
		For any finite time $T>0$, it holds
		$$ \lim\limits_{m>n\rightarrow \infty} \E \Big[\sup_{t\leq T} \mathcal{N}_{m,n} (t)\Big]=0;$$
		namely, $\{(u_n,v_n)\}_{n}$ is a Cauchy sequence in $L^2\big(\Omega;C([0,T]; L^2\times H^{-1})\big)$. Hence there exists a limit $(u,v)$ in this space, such that
		$$\lim_{ n\rightarrow \infty} \E \Big[\sup_{t\leq T} \Big(\|u_n(t)-u(t)\|_{L^2}^2+ \|v_n(t)-v(t)\|_{H^{-1}}^2\Big)\Big]=0.$$
	\end{lemma}
	
	\begin{proof}
		To begin with, we estimate $d\mathcal{N}_{m,n}$. First we have $d\|w_{m,n}(t)\|_{L^2}^2=2\<w_{m,n},z_{m,n}\> \,dt$.
		Moreover, by It\^o formula, it holds
		$$\aligned
		&\quad d\|z_{m,n}\|_{H^{-1}}^2=d \big\|(I-\Delta)^{-\frac{1}{2}} z_{m,n}\big\|_{L^2}^2\\
		&=2\big\<(I-\Delta)^{-\frac{1}{2}} z_{m,n}, (I-\Delta)^{-\frac{1}{2}} d z_{m,n}\big\>+\big\<d (I-\Delta)^{-\frac{1}{2}} z_{m,n}, d(I-\Delta)^{-\frac{1}{2}}   z_{m,n}\big\>\\
		&=2\<(I-\Delta)^{-1} z_{m,n}, \Delta w_{m,n}\> \, dt+2\big\<(I-\Delta)^{-1} z_{m,n},\Pi_m f(u_m)-\Pi_n f(u_n)\big\>\, dt\\
		&\quad+ 2\sum_k\big\<(I-\Delta)^{-1} z_{m,n}, \Pi_m (\sigma_k \cdot \nabla u_m) -\Pi_n (\sigma_k \cdot \nabla u_n)  \big\> \, dB_t^k\\
		&\quad+\sum_k \big\|(I-\Delta)^{-\frac{1}{2}}\big(\Pi_m (\sigma_k \cdot \nabla u_m) -\Pi_n (\sigma_k \cdot \nabla u_n)\big) \big\|_{L^2}^2 \, dt.
		\endaligned $$
		We will make further estimates for each finite variation term. Notice that
		$$\aligned
		\big\<(I-\Delta)^{-1} z_{m,n}, \Delta w_{m,n}\big\>&=\big\<(I-\Delta)^{-1} z_{m,n}, (\Delta-I) w_{m,n}\big\>+\big\<(I-\Delta)^{-1} z_{m,n}, w_{m,n}\big\>\\
		&=-\<z_{m,n}, w_{m,n}\>+\big\<(I-\Delta)^{-1} z_{m,n}, w_{m,n}\big\>,
		\endaligned$$
		where the former part can cancel out with $d\|w_{m,n}(t)\|_{L^2}^2$ and the latter one can be controlled as
		$$\Big|\big\<(I-\Delta)^{-1} z_{m,n}, w_{m,n}\big\>\Big| \leq \big\|(I-\Delta)^{-1} z_{m,n}\big\|_{L^2} \|w_{m,n}\|_{L^2} \lesssim \|z_{m,n}\|_{H^{-2}}^2+\|w_{m,n}\|_{L^2}^2 \leq \mathcal{N}_{m,n}(t).$$
		
		For the nonlinear term, Cauchy inequality yields
		$$\aligned
		&\quad \big|\big\<(I-\Delta)^{-1}  z_{m,n},\Pi_m f(u_m)-\Pi_n f(u_n)\big\>\big|\\
		&=\big|\big\<(I-\Delta)^{-1}  z_{m,n},(\Pi_m-\Pi_n) f(u_m)+\Pi_n \big(f(u_m)-f(u_n)\big)\big\>\big|\\
		&\leq \big\|(I-\Delta)^{-1} z_{m,n}\big\|_{L^2} \Big(\big\|(\Pi_m-\Pi_n) f(u_m)\big\|_{L^2}+\big\|f(u_m)-f(u_n)\big\|_{L^2}\Big).
		\endaligned $$
		Since $f$ is Lipschitz, we can easily deduce
		$$\|f(u_m)-f(u_n)\|_{L^2}^2 \leq C\|w_{m,n}\|_{L^2}^2.$$
		Besides, we make use of the definition of finite dimensional projection to obtain
		$$\big\|(\Pi_m-\Pi_n) f(u_m)\big\|_{L^2}^2=\sum_{n<|l|\leq m} \big|\<f(u_m),e_l\>\big|^2=\sum_{n<|l|\leq m} \frac{1}{|l|^2} |l|^{2}\big|\<f(u_m),e_l\>\big|^2\leq \frac{1}{n^2} \|f(u_m)\|_{H^1}^2.$$
		By the definition of Sobolev norm, we can further estimate $\|f(u_m)\|_{H^1}^2$ as
		$$\aligned
		\|f(u_m)\|_{H^1}^2&=\int_{\T^d} \Big(\big|f(u_m)\big|^2+\big|\nabla f(u_m)\big|^2\Big) \, dx \\
		&\leq \int_{\T^d} \Big[C^2\big(1+|u_m|^2\big)+\big|f'(u_m)\big|^2 |\nabla u_m|^2\Big] \, dx\\
		&\leq C \int_{\T^d} \Big(1+|u_m|^2+|\nabla u_m|^2\Big) \, ds\\
		&\leq C\big(1+\|u_m\|_{H^1}^2\big).
		\endaligned $$
		Summarizing the above estimates, Young's inequality yields
		$$\aligned
		\big|\big\<(I-\Delta)^{-1} z_{m,n},\Pi_m f(u_m)-\Pi_n f(u_n)\big\> \big| &\leq C \Big(\big\|(I-\Delta)^{-1} z_{m,n}\big\|_{L^2}^2+\|w_{m,n}\|_{L^2}^2+\frac{1+\|u_m\|_{H^1}^2}{n^2}\Big)\\
		&\leq C\Big(\mathcal{N}_{m,n}(t)+\frac{1+\|u_m\|_{H^1}^2}{n^2}\Big).
		\endaligned $$
		
		Then we use triangle inequality to estimate the quadratic variation term:
		$$\aligned
		&\quad \sum_k \big\|(I-\Delta)^{-\frac{1}{2}} \big(\Pi_m (\sigma_k \cdot \nabla u_m) -\Pi_n (\sigma_k \cdot \nabla u_n)\big)  \big\|_{L^2}^2 \\
		&\leq 2\sum_k \big\|(I-\Delta)^{-\frac{1}{2}}  \big((\Pi_m-\Pi_n)(\sigma_k \cdot \nabla u_m)\big) \big\|_{L^2}^2+2\sum_k \big\|(I-\Delta)^{-\frac{1}{2}}  \Pi_n\big(\sigma_k \cdot \nabla (u_m-u_n)\big) \big\|_{L^2}^2.
		\endaligned $$
		For the former part, it holds
		$$\aligned
		\sum_k \big\|(I-\Delta)^{-\frac{1}{2}}  \big((\Pi_m-\Pi_n)(\sigma_k \cdot \nabla u_m)\big) \big\|_{L^2}^2 &=\sum_k \sum_{n<|l|\leq m} \big|\big\<\sigma_k \cdot \nabla u_m,(I-\Delta)^{-\frac{1}{2}}  e_l\big\>\big|^2\\
		&\lesssim \sum_k \sum_{n<|l|\leq m} (1+|l|^2)^{-1} \big|\big\<\sigma_k \cdot \nabla u_m,  e_l\big\>\big|^2\\
		&\leq \frac{1}{1+n^2} \sum_k  \|\sigma_k \cdot \nabla u_m\|_{L^2}^2 \lesssim_\kappa \frac{1}{n^2} \|\nabla u_m\|_{L^2}^2,
		\endaligned $$
		where the last step follows similarly to \eqref{eq-Q L2}.
		For the latter part, recalling that $\{\sigma_k\}_k$ are divergence-free,  we have
		\begin{equation}\label{eq-star}
			\sum_k \big\|(I-\Delta)^{-\frac{1}{2}}  \Pi_n(\sigma_k \cdot \nabla (u_m-u_n)) \big\|_{L^2}^2 \leq \sum_k \|\sigma_k \cdot \nabla w_{m,n}\|_{H^{-1}}^2 \leq \sum_k \|\sigma_k w_{m,n}\|_{L^2}^2 \lesssim_\kappa \|w_{m,n}\|_{L^2}^2,
		\end{equation}
		and therefore we arrive at
		\begin{equation}\label{eq-quadratic estimate}
			\sum_k \big\|(I-\Delta)^{-\frac{1}{2}} \big(\Pi_m (\sigma_k \cdot \nabla u_m) -\Pi_n (\sigma_k \cdot \nabla u_n)\big)  \big\|_{L^2}^2 \leq C\Big(\mathcal{N}_{m,n}(t)+\frac{\| u_m\|_{H^1}^2}{n^2}\Big).
		\end{equation}
		Summarizing the above estimates and recalling the definition of $\mathcal{N}_{m,n}$, we have
		$$\aligned
		d\mathcal{N}_{m,n}(t)&\leq C\Big(\mathcal{N}_{m,n}(t)+\frac{1+\|u_m\|_{H^1}^2}{n^2}\Big) \, dt\\
		&\quad+2\sum_k\big\<(I-\Delta)^{-1}z_{m,n}, \Pi_m (\sigma_k \cdot \nabla u_m) -\Pi_n (\sigma_k \cdot \nabla u_n)  \big\> \, dB_t^k.
		\endaligned$$
		Integrating it with respect to time, we take supremum and then expectation to get
		\begin{equation}\label{eq-supEmn}
			\begin{split}
				\E \Big[\sup_{t\leq T} \mathcal{N}_{m,n}(t) \Big]&\leq \mathcal{N}_{m,n}(0)+C\,  \E \bigg[\int_{0}^{T}  \Big(\mathcal{N}_{m,n}(t)+\frac{1+\|u_m(t)\|_{H^1}^2}{n^2}\Big) \, dt\bigg]\\
				&+2\E \bigg[\sup_{t\leq T} \Big|\sum_k \int_{0}^t \big\<(I-\Delta)^{-1}z_{m,n}, \Pi_m (\sigma_k \cdot \nabla u_m) -\Pi_n (\sigma_k \cdot \nabla u_n)  \big\> \, dB_s^k\Big| \bigg].
			\end{split}
		\end{equation}
		Denote the last term by $J_{m,n}$. Applying Burkholder-Davis-Gundy inequality, we have
		$$\aligned
		J_{m,n}\leq C \, \E \bigg[\bigg(\int_{0}^{T} \sum_k \big\<(I-\Delta)^{-1}z_{m,n}, \Pi_m (\sigma_k \cdot \nabla u_m) -\Pi_n (\sigma_k \cdot \nabla u_n)  \big\>^2 \, dt\bigg)^\frac{1}{2}\bigg].
		\endaligned $$
		By Cauchy inequality and \eqref{eq-quadratic estimate}, it holds
		$$\aligned
		&\quad \sum_k \big\<(I-\Delta)^{-1}z_{m,n}, \Pi_m (\sigma_k \cdot \nabla u_m) -\Pi_n (\sigma_k \cdot \nabla u_n)  \big\>^2\\
		&\leq \sum_k \big\|(I-\Delta)^{-\frac{1}{2}} z_{m,n}\big\|_{L^2}^2 \big\|(I-\Delta)^{-\frac{1}{2}} \big(\Pi_m (\sigma_k \cdot \nabla u_m) -\Pi_n (\sigma_k \cdot \nabla u_n)\big)  \big\|_{L^2}^2 \\
		&\leq C \big\|z_{m,n}(t)\big\|_{H^{-1}}^2 \Big(\mathcal{N}_{m,n}(t)+\frac{\|u_m(t)\|_{H^1}^2}{n^2}\Big).
		\endaligned $$
		Substituting it into the above estimate, we obtain, for a small $\varepsilon>0$,
		$$\aligned
		J_{m,n}&\leq C \, \E \bigg[\bigg(\int_{0}^{T} \big\| z_{m,n}(t)\big\|_{H^{-1}}^2 \Big(\mathcal{N}_{m,n}(t)+\frac{\| u_m(t)\|_{H^1}^2}{n^2}\Big) \, dt\bigg)^\frac{1}{2}\bigg]\\
		&\leq C \, \E \bigg[\bigg( \sup_{t\leq T} \|z_{m,n}(t)\|_{H^{-1}}^2\int_{0}^{T} \Big(\mathcal{N}_{m,n}(t)+\frac{\|u_m(t)\|_{H^1}^2}{n^2}\Big) \, dt\bigg)^\frac{1}{2}\bigg]\\
		&\leq \varepsilon \, \E \Big[\sup_{t\leq T} \mathcal{N}_{m,n}(t)\Big]+C(\varepsilon) \, \E \bigg[\int_0^T \sup_{s\leq t} \mathcal{N}_{m,n}(s) \, dt\bigg]+C(\varepsilon) \, \E \bigg[\int_{0}^T \frac{\|u_m(t)\|_{H^1}^2}{n^2}\, dt\bigg].
		\endaligned $$
		Recalling \eqref{eq-supEmn}, we can deduce the following estimate:
		$$\aligned
		\E \Big[\sup_{t\leq T} \mathcal{N}_{m,n}(t) \Big]&\leq C \bigg(\mathcal{N}_{m,n}(0)+ \int_{0}^{T}  \E\Big[ \sup_{s\leq t} \mathcal{N}_{m,n}(s) \Big] \, dt+\E \int_{0}^T \frac{1+\|u_m(t)\|_{H^1}^2}{n^2}\, dt\bigg).
		\endaligned $$
		Now we can apply Gr\"onwall's inequality to obtain
		$$\E \Big[\sup_{t\leq T} \mathcal{N}_{m,n}(t) \Big] \leq C_T \bigg(\mathcal{N}_{m,n}(0)+\E \int_{0}^T  \frac{1+\|u_m(t)\|_{H^1}^2}{n^2} \, dt\bigg).$$
		Notice that $\lim\limits_{m>n\rightarrow \infty} \mathcal{N}_{m,n}(0)=0$; and thanks to Lemma \ref{lemma-energy estimate}, one has
		$$\lim\limits_{m>n\rightarrow \infty} \E \int_{0}^T  \frac{1+\|u_m(t)\|_{H^1}^2}{n^2} \, dt=0,$$
		which gives the first result of this lemma. Finally, the completeness of $L^2\big(\Omega;C([0,T]; L^2\times H^{-1})\big)$ ensures the existence of a limit $(u,v)$ of the Cauchy sequence $\{(u_n,v_n)\}_n$.
	\end{proof}
	
	\begin{lemma}\label{lemma-uv regularity}
		The limit stochastic process $(u,v)$ mentioned in Lemma \ref{lemma-Cauchy sequence for SWE1} belongs to the space $L^2(\Omega;L^\infty([0,T];H^1 \times L^2))$ and satisfies the following estimate:
		$$\E \Big[ \|u\|_{L^\infty([0,T];H^1)}^2+\|v\|_{L^\infty([0,T];L^2)}^2\Big] \leq C_T\big(\mathcal{E}(u(0))+1\big).$$
	\end{lemma}
	\begin{proof}
		By Lemma \ref{lemma-energy estimate}, $\{(u_n,v_n)\}_n$ is uniformly  bounded in $L^2(\Omega,L^\infty([0,T];H^1 \times L^2))$, which is the dual space of $L^2(\Omega,L^1([0,T];H^{-1} \times L^2))$. By Banach-Alaoglu Theorem, there exist a subsequence  $\{(u_{n_j},v_{n_j})\}_j$ and a limit $(u^\ast, v^\ast) $, such that $\{(u_{n_j},v_{n_j})\}_j$ weak-$\ast$ converges to $(u^\ast, v^\ast)$ in $L^2(\Omega,L^\infty([0,T];H^1 \times L^2))$.
		On the other hand, by Lemma \ref{lemma-Cauchy sequence for SWE1}, $\{(u^n,v^n)\}_n$ converges strongly to $(u,v)$ in the space $L^2(\Omega;C([0,T];L^2\times H^{-1}))$. Hence we can deduce $(u^\ast, v^\ast)=(u,v)$ for a.e. $t\in [0,T]$, implying that the latter pair belongs to $L^2(\Omega,L^\infty([0,T];H^1 \times L^2))$.
		Finally, the norm is weak-$\ast$ lower semi-continuous, so we obtain
		$$\|(u,v)\|_{L^2(\Omega;L^\infty([0,T];H^1\times L^2))} \leq \liminf_{j\rightarrow \infty} \|(u_{n_j},v_{n_j})\|_{L^2(\Omega;L^\infty([0,T];H^1))}\leq C_T^{1/2}\big(\mathcal{E}(u(0))+1\big)^{1/2}.$$
		Taking square gives us the claimed estimate.
	\end{proof}

	Now we can present the proof of Theorem \ref{thm-SWE1 wellposedness}. We first show the limit $(u,v)$ is a solution to \eqref{eq-SWE}, yielding the existence of weak solutions; then we establish the pathwise uniqueness.
	\begin{proof}[Proof of Theorem \ref{thm-SWE1 wellposedness}.]
		\textbf{Existence.} For each $n\geq 1$, the solutions to \eqref{eq-finite dimension} satisfy the following identity for any $\phi \in C^\infty(\T^d)$:
		$$\aligned
		\<u_n(t),\phi\>&=\< \Pi_n u_0, \phi\>+\int_{0}^{t} \<v_n(s),\phi\> \, ds, \\
		\<v_n(t),\phi\>&=\<\Pi_n v_0,\phi\>+\int_{0}^{t} \<u_n(s), \Delta \phi \> \, ds+\int_{0}^{t} \<\Pi_n f(u_n(s)),\phi\> \, ds-\sum_k \int_{0}^{t} \<u_n(s),\sigma_k \cdot \nabla \Pi_n \phi \> \, dB^k_s.
		\endaligned $$
		If we take $n\rightarrow \infty$ in the above equations and prove that each term converges to the corresponding one of identity \eqref{eq-def of weak sol}, then the existence follows.  In fact, the proofs for linear terms can be done by Lemma \ref{lemma-Cauchy sequence for SWE1}; as for the arguments for martingale term, one can find similar proofs in \cite[Theorem 2.2]{FGL21 JEE}. So we are mainly concerned with the nonlinear term here and we will show
		$$\E \bigg[\sup_{t\in [0,T]}\bigg|\int_{0}^{t} \<\Pi_n f(u_n),\phi\> \, ds-\int_{0}^{t} \<f(u),\phi\> \, ds\bigg| \bigg] \rightarrow 0.$$
		By triangle inequality, the left-hand side can be controlled by
		$$\aligned
		&\quad \E \bigg[\int_{0}^{T} \big|\big\<\Pi_n\big(f(u_n)-f(u)\big), \phi \big\>\big| \, ds\bigg]+\E \bigg[\int_{0}^{T} \big|\big\<(\Pi_n-I) f(u),\phi \big\>\big| \, ds\bigg]\\
		&\leq \|\phi\|_{L^2} \, \E \bigg[\int_{0}^{T} \|f(u_n)-f(u)\|_{L^2} \, ds\bigg]  + \|(\Pi_n-I)\phi\|_{L^2} \, \E \bigg[\int_{0}^{T}  \| f(u)\|_{L^2} \, ds\bigg].
		\endaligned $$
		Since we have assumed that $f$ is Lipschitz, it holds $\|f(u_n)-f(u)\|_{L^2} \leq C\|u_n-u\|_{L^2}$, and therefore the first term vanishes as $n\rightarrow \infty$ by Lemma \ref{lemma-Cauchy sequence for SWE1}. Besides, we have $\|f(u)\|_{L^2} \leq C(1+\|u\|_{L^2}^2)^\frac{1}{2}$, which is bounded due to Lemma \ref{lemma-uv regularity}, hence we can get the convergence of the second part. Summarizing the above arguments, we obtain the existence.
		
		\textbf{Pathwise uniqueness.} Suppose $(u^1,v^1)$ and $(u^2,v^2)$ are two solutions of \eqref{eq-SWE} with the same initial value, and they are defined on the same probability space and driven by the same Brownian motions. Then we denote $(w,z):=(u^1-u^2,v^1-v^2)$ which satisfies
		\begin{equation}\nonumber
			\left\{
			\begin{aligned}
				&\partial_t w=z,  \\
				& d z=\Delta w \, dt+\big(f(u^1)-f(u^2) \big)\, dt+\sum_k \sigma_k \cdot \nabla w \, dB_t^k;
			\end{aligned}
			\right.
		\end{equation}
		we aim to show that $(w,z) \equiv 0$. The idea is to apply It\^o's formula to the energy $\mathcal{N}(t)= \|w(t)\|_{L^2}^2+ \|z(t)\|_{H^{-1}}^2$; to this end, by \cite[Theorem 2.13]{RL 1990}, we need to verify the following conditions:
		
		(1) The following term is a continuous martingale taking values in $L^2(\T^d)$:
		$$M_t:=\sum_{k}  \int_{0}^{t} (I-\Delta)^{-\frac{1}{2}}\big(\sigma_{k} \cdot \nabla w(s)\big) \, dB_s^k.$$
		
		(2) For $x(t)=(I-\Delta)^{-\frac{1}{2}} z$ and $x'(t)=(I-\Delta)^{-\frac{1}{2}} \Delta w+(I-\Delta)^{-\frac{1}{2}} \big(f(u^1)-f(u^2)\big)$, it holds
		$$\int_{0}^{T}  \|x(t)\|_{H^1}^2+\|x'(t)\|_{H^{-1}}^2 \, dt<+\infty, \quad \P \text{-a.s.}.$$
		The first one can be proved as
		$$ \E \|M_t\|_{L^2}^2 = \E  \int_{0}^{t} \sum_{k}  \big\|(I-\Delta)^{-\frac{1}{2}}\big(\sigma_{k} \cdot \nabla w(s)\big)\big\|^2_{L^2} \, ds  =\E\int_{0}^{t} \sum_{k} \big\|\sigma_{k} \cdot \nabla w(s)\big\|_{H^{-1}}^2 \, ds.$$
		Recalling that $\{\sigma_k\}_k$ are divergence-free and $w \in L^2(\Omega,L^\infty([0,T];H^1))$, we can further obtain
		$$\E \|M_t\|_{L^2}^2  \leq \E\int_{0}^{t} \sum_{k} \big\|\sigma_{k} w(s)\big\|_{L^2}^2 \, ds \lesssim_\kappa \E\int_{0}^{t} \big\|w(s)\big\|_{L^2}^2 \, ds<+\infty.$$
		This verifies condition (1). As for the second condition, we can utilize the Lipschitz property of $f$ to obtain, $\P$-a.s.,
		$$\aligned
		&\quad \int_{0}^{T}  \|x(t)\|_{H^1}^2+\|x'(t)\|_{H^{-1}}^2 \, dt \\
		&\leq \int_{0}^{T} \|z\|_{L^2}^2+\big\|(I-\Delta)^{-\frac{1}{2}}\Delta w\big\|_{H^{-1}}^2+\big\|(I-\Delta)^{-\frac{1}{2}} \big(f(u^1)-f(u^2)\big) \big\|_{H^{-1}}^2 \, dt\\
		&\lesssim \int_{0}^{T} \|z\|_{L^2}^2+\|w\|_{L^2}^2+\big\|f(u^1)-f(u^2)\big\|_{L^{2}}^2 \, dt\\
		&\lesssim \int_{0}^{T} \|z\|_{L^2}^2+\|w\|_{L^2}^2 \, dt<+\infty,
		\endaligned $$
		where in the last step we used $(w,z)\in L^2(\Omega,L^\infty([0,T];H^1\times L^2))$.
		
		Having verified the conditions (1)-(2), now we can apply It\^o formula to obtain
		$$\aligned
		d\mathcal{N}(t)&=2\<w, z \> \, dt+2\big\<(I-\Delta)^{-1}z, dz\big\>+\big\<d(I-\Delta)^{-\frac{1}{2}} z,d(I-\Delta)^{-\frac{1}{2}} z \big\>\\
		&=2\<w, z \> \, dt+2\big\<(I-\Delta)^{-1} z, \Delta w \big\> \, dt+2 \big\<(I-\Delta)^{-1}z, f(u^1)-f(u^2)\big\> \, dt\\
		&\quad +2\sum_k \big\<(I-\Delta)^{-1}z, \sigma_k \cdot \nabla w  \big\> \, dB_t^k+\sum_k \big\|(I-\Delta)^{-\frac{1}{2}} ( \sigma_k \cdot \nabla w )\big\|_{L^2}^2 \, dt.
		\endaligned $$
		Noticing that $\<(I-\Delta)^{-1} z, \Delta w \>=\<(I-\Delta)^{-1} z, (\Delta-I) w \> +\<(I-\Delta)^{-1} z,  w \>$, it holds
		$$\<w, z \> +\big\<(I-\Delta)^{-1} z, \Delta w \big\>  =\big\<(I-\Delta)^{-1} z,  w \big\>;$$
		moreover, since $f$ is Lipschitz, we have $\|f(u^1)-f(u^2)\|_{L^2} \lesssim \|u^1-u^2\|_{L^2}=\|w\|_{L^2}$. For the quadratic variation term, we can apply similar computations in \eqref{eq-star} to obtain
		$$\sum_k \big\|(I-\Delta)^{-\frac{1}{2}} ( \sigma_k \cdot \nabla w )\big\|_{L^2}^2  \lesssim_\kappa \|w\|_{L^2}^2.$$
		Summarizing these estimates and by Cauchy inequality, we arrive at
		$$\aligned
		d\mathcal{N}(t) &\lesssim_\kappa \big\| (I-\Delta)^{-1} z \big\|_{L^2} \|w\|_{L^2}\, dt + \sum_k \big\<(I-\Delta)^{-1}z, \sigma_k \cdot \nabla w  \big\> \, dB_t^k + \|w\|_{L^2}^2\, dt \\
		&\leq C \mathcal{N}(t) \, dt+\sum_k \big\<(I-\Delta)^{-1}z, \sigma_k \cdot \nabla w  \big\> \, dB_t^k,
		\endaligned $$
		where the last step follows from Young's inequality and the definition of $\mathcal{N}(t)$.
		Integrating from $0$ to $t$ and recalling that $\mathcal{N}(0)=0$, we get
		$$\mathcal{N}(t) \leq C\int_{0}^{t} \mathcal{N}(s) \, ds+\int_{0}^{t} \sum_k \big\<(I-\Delta)^{-1}z, \sigma_k \cdot \nabla w  \big\> \, dB_s^k.$$
		Recalling that $w \in L^2(\Omega,L^\infty([0,T];H^1))$, we can verify that the stochastic integral is a square integrable martingale, hence we take expectation to obtain
		$$\E \, \mathcal{N}(t) \leq C \, \E \int_{0}^{t} \mathcal{N}(s) \, ds.$$
		Then Gr\"onwall's lemma yields $\E \mathcal{N}(t)= 0$ for any $t\in [0,T]$. Noticing that $\mathcal{N}(t) \geq 0$, we obtain $\mathcal{N}(t)=0$, $\P$-a.s. and therefore the pathwise uniqueness follows.
	\end{proof}
	
	\section{Weak existence for \eqref{eq-SWE2}}\label{sec-prf of SWE2 wellposed}
	This section is devoted to proving Theorem \ref{thm-SWE2 wellposedness}.
	We still consider the finite dimensional approximation of \eqref{eq-SWE2}:
	\begin{equation}\label{eq-SWE2 approximation}
		\left\{
		\begin{aligned}
			&\partial_t u_n=v_n,  \\
			& d v_n=\Delta u_n \, dt+\Pi_n f(u_n) \, dt+\sum_k \Pi_n(\sigma_k \cdot \nabla v_n) \, dB_t^k+\kappa \Delta v_n \, dt,
		\end{aligned}
		\right.
	\end{equation}
	with the initial data projected accordingly: $(u_n(0),v_n(0))=(\Pi_n u_0,\Pi_n v_0)$.
	We explain in the following remark that the approach used in Theorem \ref{thm-SWE1 wellposedness} cannot be directly applied here.
	\begin{remark}\label{rmk-SWE2 not Cauchy}
		The method of Lemma \ref{lemma-Cauchy sequence for SWE1} does not work to show that the above sequence $\{(u_n,v_n)\}_n$ is Cauchy in $L^2(\Omega;C([0,T]; H^{-\alpha}\times H^{-\alpha-1}))$ for any $\alpha \in [0,1)$. Indeed, when estimating $d\|v_m-v_n\|_{H^{-\alpha-1}}^2$, there exists a quadratic variation term
		$$\mathcal{S}_{m,n}(v):=\sum_{k} \Big\|(I-\Delta)^{-\frac{\alpha+1}{2}} \big(\Pi_m(\sigma_{k} \cdot \nabla v_m)-\Pi_n(\sigma_{k} \cdot \nabla v_n)\big)\Big\|_{L^2}^2,$$
		which, by the triangle inequality, can be decomposed as two parts:
		$$ \sum_{k} \Big\|(I-\Delta)^{-\frac{\alpha+1}{2}} (\Pi_m-\Pi_n)(\sigma_{k} \cdot \nabla v_m)\Big\|_{L^2}^2+\sum_{k}\Big\|(I-\Delta)^{-\frac{\alpha+1}{2}} \Pi_n \big(\sigma_{k} \cdot \nabla (v_m-v_n)\big) \Big\|_{L^2}^2.$$
		We denote these two terms by $\mathcal{S}_{m,n}^1(v)$ and $\mathcal{S}_{m,n}^2(v)$ respectively.
		
		To apply Gr\"onwall's lemma later, we need $\mathcal{S}_{m,n}^1(v)\rightarrow 0$ as $m>n\rightarrow \infty$ and $\mathcal{S}_{m,n}^2(v)$ to be controlled by $\|v_m-v_n\|_{H^{-\alpha-1}}^2$.
		While the former property can be verified with our current estimates, we are not able to establish the latter. This prevents us from identifying a suitable space in which the sequence $\{(u_n,v_n)\}_n$ is Cauchy.
	\end{remark}
	
	Hence we will apply the compactness method, which relies on Theorem \ref{thm-compact embedding}, to prove the existence of weak solutions. We first establish a pathwise regularity estimate for equation \eqref{eq-SWE2 approximation}.
	\begin{lemma}\label{lemma-energy 2}
		Suppose $(u_0,v_0) \in H^1\times L^2$, then for any $T>0$, it holds $\P$-a.s.,
		\begin{equation}\label{eq-uniform estimate for SWE 2}
			\sup_{n\geq 1} \sup_{t\in [0,T]} \mathcal{E}(u_n(t)) \leq C_T\big(\mathcal{E}(u(0))+1\big)=:C_{ub}^2,
		\end{equation}
		where $\mathcal{E}(u_n(t))$ is defined as in \eqref{eq-def of energy} for the solutions $(u_n,v_n)$ of equation \eqref{eq-SWE2 approximation} and $\mathcal{E}(u(0))$ denotes the corresponding initial energy for \eqref{eq-SWE2}.
	\end{lemma}
	
	\begin{proof}
		Fix any $n\geq 1$ and we will make energy estimate below. As shown before, it holds
		$$d\|u_n\|_{H^1}^2=d\big\|(I-\Delta)^{\frac{1}{2}} u_n\big\|_{L^2}^2=2\<u_n,v_n\>\, dt-2\<\Delta u_n, v_n\> \,dt.$$
		Moreover, by It\^o formula, we have
		$$\aligned
		d\|v_n\|_{L^2}^2&=2\<v_n,dv_n\>+\<dv_n,dv_n\>\\
		&=2\<v_n,\Delta u_n\> \, dt+2\big\<v_n,\Pi_n f(u_n)\big\> \, dt+2\sum_{k} \big\<v_n, \Pi_n(\sigma_k \cdot \nabla v_n)\big\> dB_t^k\\
		&\quad + 2\kappa \<v_n,\Delta v_n\> \, dt+ \sum_{k} \big\|\Pi_n(\sigma_k \cdot \nabla v_n)\big\|_{L^2}^2 \, dt\\
		&\leq 2\<v_n,\Delta u_n\> \, dt+2\<v_n,f(u_n)\> \, dt,
		\endaligned $$
		where the last step follows from the divergence-free property of $\{\sigma_k\}_k$ and similar computations as those in \eqref{eq-Q L2}.
		Recalling that we have assumed that $f$ is Lipschitz, hence Cauchy and Young inequalities yield
		$$d\mathcal{E}(u_n(t)) \leq 2 \|v_n\|_{L^2} \big(\|f(u_n)\|_{L^2}+\|u_n\|_{L^2}\big) \, dt \leq C \big(\mathcal{E}(u_n(t))+1\big) \, dt. $$
		Integrating the above inequality and applying Gr\"onwall's lemma, we obtain
		$$\mathcal{E}(u_n(t)) \leq C_T\big(\mathcal{E}(u_n(0))+1\big) \leq C_T\big(\mathcal{E}(u(0))+1\big).$$
		Since the right-hand side is independent of $n\geq 1$ and $t\in [0,T]$, we can take supremum and therefore complete the proof.
	\end{proof}
	
	The above lemma implies that $\P$-a.s., $\{(u_n,v_n)\}_n \subset L^\infty([0,T];H^1\times L^2)$; actually, one can directly deduce that $\{u_n\}_n$ also belongs to $W^{1,\infty} ([0,T];L^2)$ due to the relationship $\partial_t u_n=v_n$. Hence we can apply the point (1) in Theorem \ref{thm-compact embedding}.
	
	In order to utilize the second point of Theorem \ref{thm-compact embedding}, we need to show that, for $\rho >\frac{d}{2}+1$, it holds $\P$-a.s. that $\{v_n\}_n \subset W^{\frac{1}{3},4}([0,T];H^{-\rho})$, with the corresponding norm given by
	$$\|v_n\|_{W^{\frac{1}{3},4} ([0,T];H^{-\rho})}^4=\|v_n\|_{L^4([0,T];H^{-\rho})}^4+\int_{0}^{T} \int_{0}^{T} \frac{\|v_n(t)-v_n(s)\|_{H^{-\rho}}^4}{|t-s|^{\frac{7}{3}}} \, dtds.$$
	Thanks to Lemma \ref{lemma-energy 2}, $\{v_n\}_n \subset L^4([0,T];H^{-\rho})$ $\P$-a.s.. So it is enough to show that, for any $n\geq 1$, the double integral is finite in the sense of expectation, which relies on the following estimate.
	\begin{lemma}\label{lemma-fractional sobolev estimate}
		For any $T>0$ and $\rho>\frac{d}{2}+1$, we can find a finite constant $C$ depending on $T, \kappa$ and $C_{ub}$ but independent of $n$, such that for any $0\leq s<t \leq T$, the approximated sequence satisfies
		$$\sup_{n\geq 1} \E\|v_n(t)-v_n(s)\|_{H^{-\rho}}^4 \leq C |t-s|^2.$$
	\end{lemma}
	\begin{proof}
		We fix any $n\geq 1$. For any $0\leq s<t \leq T$, integrating equation \eqref{eq-SWE2 approximation} yields
		$$\aligned
		\|v_n(t)-v_n(s)\|_{H^{-\rho}} &\leq \int_{s}^{t} \|u_n(r)\|_{H^{-\rho+2}} \, dr+ \int_{s}^{t} \big\|f(u_n(r))\big\|_{H^{-\rho}} \, dr\\
		&+ \bigg\| \sum_{k} \int_{s}^{t} \Pi_n (\sigma_{k} \cdot \nabla v_n(r)) \, dB_r^k\bigg\|_{H^{-\rho}} +\kappa \int_{s}^{t} \|v_n(r)\|_{H^{-\rho+2}} \, dr.
		\endaligned$$
		We denote four terms on the right-hand side of the above inequality as $I_1, I_2, I_3, I_4$ respectively. Since $\rho>\frac{d}{2}+1\geq 2$, we apply H\"older's inequality to obtain
		$$\E I_1^4 \leq |t-s|^3 \, \E \int_{s}^{t}  \|u_n(r)\|_{H^{-\rho+2}}^4 \, dr \leq C_{ub}^4 \, |t-s|^4.$$
		Similar estimates can be applied to $I_4$. For the term $I_2$, recalling that $f$ is Lipschitz, we have
		$$\E I_2^4 \leq |t-s|^3 \, \E \int_{s}^{t}  \big\|f(u_n(r))\big\|_{L^2}^4 \, dr \lesssim |t-s|^3 \, \E \int_{s}^{t} \big(1+\|u_n(r)\|_{L^2}^2\big)^2 \, dr \lesssim \big(1+2C_{ub}^2+C_{ub}^4\big) |t-s|^4. $$
		Finally, we can apply Burkholder-Davis-Gundy inequality to estimate $I_3$:
		$$\E I_3^4 \lesssim \, \E \bigg[ \Big|\sum_{k}  \int_{s}^{t} \big\| \Pi_n(\sigma_{k} \cdot \nabla v_n(r))\big\|_{H^{-\rho}}^2  \, dr\Big|^2 \bigg].$$
		By the definition of Sobolev norm and the divergence-free property of  $\sigma_{k}$, it holds
		$$\aligned
		\sum_{k} \big\|\sigma_k \cdot \nabla v_n(r)\big\|_{H^{-\rho}}^2
		&= \sum_{k} \sum_{l \in \Z^d} \big(1+|l|^2\big)^{-\rho}
		\big| \< \sigma_k \cdot\nabla v_n(r), e_l \>\big|^2\\
		&= \sum_{l \in \Z^d} \big(1+|l|^2\big)^{-\rho}\sum_{k} \big| \< v_n(r), \sigma_k\cdot \nabla e_l \> \big|^2.
		\endaligned $$
		By Cauchy inequality and the assumption \eqref{eq-asp on Q}, we have
		$$\sum_{k} \big| \< v_n(r), \sigma_k \cdot \nabla e_l \> \big|^2\leq \|v_n(r)\|_{L^2}^2 \sum_{k} \|\sigma_k \cdot \nabla e_l\|_{L^2}^2=2\kappa \|v_n(r)\|_{L^2}^2 \|\nabla e_l\|_{L^2}^2.$$
		Substituting it into the above inequality yields
		$$\sum_{k} \big\|\sigma_k \cdot \nabla v_n(r)\big\|_{H^{-\rho}}^2 \lesssim_\kappa \|v_n(r)\|_{L^2}^2 \sum_{l \in \Z^d} \big(1+|l|^2\big)^{-\rho} \|\nabla e_l\|_{L^2}^2 \lesssim  \|v_n(r)\|_{L^2}^2 \sum_{l \in \Z^d} \frac{|l|^2}{(1+|l|^2)^{\rho}}. $$
		Since $\rho>d/2+1$, the sum with respect to $l$  is finite and we finally arrive at
		$$\sum_{k}  \big\| \Pi_n(\sigma_{k} \cdot \nabla v_n(r))\big\|_{H^{-\rho}}^2  \lesssim  \|v_n(r)\|_{L^2}^2.$$
		Hence we apply Lemma \ref{lemma-energy 2} to arrive at $\E I_3^4 \lesssim C_{ub}^4 \, |t-s|^2 $. Summarizing the estimates for $I_i, i=1,\ldots,4$ and noticing that $|t-s|\leq T$, we complete the proof.
	\end{proof}
	Based on the above lemma, we can claim that
	$$\E \int_{0}^{T} \int_{0}^{T} \frac{\|v_n(t)-v_n(s)\|_{H^{-\rho}}^4}{|t-s|^{7/3}} \, dtds\leq  \int_{0}^{T} \int_{0}^{T} \frac{1}{|t-s|^\frac{1}{3}} \, dtds <+\infty,$$
	and therefore $\P$-a.s., $\{v_n\}_n \subset W^{\frac{1}{3},4}([0,T];H^{-\rho})$. Then applying Theorem \ref{thm-compact embedding}, one has
	\begin{corollary}\label{coro-tightness}
		For $n\geq 1$, we denote $\mathcal{L}_n$ as the joint law of $(u_n,v_n)$. The sequence $\{\mathcal{L}_n\}_n$ is tight on $\mathcal{X}:=C([0,T];H^{1-\gamma}\times H^{-\gamma})$, where $\gamma \in (0,\frac{1}{2})$.
	\end{corollary}

	With the above preparations in hand, we are now ready to establish the existence of weak solutions to equation \eqref{eq-SWE2}.
	\begin{proof}[Proof of Theorem \ref{thm-SWE2 wellposedness}.]
		We regard the sequence of Brownian motions $B:=\big\{B^k: k \in \mathbb{N}\big\}$ as a random variable with values in $\mathcal{Y}:= C([0,T],\mathbb{R}^{\mathbb{N}})$. Denote $P_n$ as the joint law of $(u_n, v_n, B)$ for each $n\geq 1$, we claim that  $\{P_n\}_{n}$ is tight on $\mathcal{X} \times \mathcal{Y}$.
		
		By Prohorov theorem, we can find a subsequence $\{n_j\}_{j\geq 1}$ such that $P_{n_j}$ converges weakly to some probability measure $P$ on $\mathcal{X} \times \mathcal{Y}$ as $j \rightarrow \infty$. Applying Skorohod representation theorem, there exists a new probability space $(\tilde{\Omega},\tilde{\mathcal{F}},\tilde{\P})$, on which we can define a family $\tilde{B}^{n_j}:=\big\{\tilde{B}^{n_j,k}: k\in \mathbb{N}\big\}$ of independent Brownian motions and stochastic processes $\{(\tilde{u}_{n_j}, \tilde{v}_{n_j})\}_{j\geq 1}$, together with $(\tilde{u},\tilde{v})$ and Brownian motion $\tilde{B}$, such that
		
		(1) for any $j\geq 1$, $(\tilde{u}_{n_j}, \tilde{v}_{n_j},\tilde{B}^{n_j})$ has the same joint law as $(u_{n_j}, v_{n_j}, B)$;
		
		(2) $\tilde{\P}$-a.s., $(\tilde{u}_{n_j}, \tilde{v}_{n_j}, \tilde{B}^{n_j})$ converges to the limit $(\tilde{u},\tilde{v},\tilde{B})$ in $\mathcal{X} \times \mathcal{Y}$ as $j\rightarrow \infty$.
		
		Once we can show that $(\tilde{u},\tilde{v},\tilde{B})$ satisfies the identity \eqref{eq-SWE2 sol def}, then the existence holds. Since $(\tilde{u}_{n_j}, \tilde{v}_{n_j},\tilde{B}^{n_j})$ has the same joint law as $(u_{n_j}, v_{n_j}, B)$, we conclude that the former triple satisfies the following weak formulation: for $\phi \in C^\infty$,
		\begin{equation}\label{eq-SWE2 existence appro}
			\begin{split}
				\<\tilde{u}_{n_j}(t),\phi\>&=\< \Pi_{n_j} u_0, \phi\>+\int_{0}^{t} \<\tilde{v}_{n_j}(s),\phi\> \, ds, \\
				\<\tilde{v}_{n_j}(t),\phi\>&=\<\Pi_{n_j} v_0,\phi\>+\int_{0}^{t} \<\tilde{u}_{n_j}(s), \Delta \phi \> \, ds+\int_{0}^{t} \big\< f(\tilde{u}_{n_j}(s)), \Pi_{n_j}\phi \big\> \, ds\\
				&\quad -\sum_k \int_{0}^{t} \<\tilde{v}_{n_j}(s),\sigma_k \cdot \nabla \Pi_{n_j} \phi \> \, d\tilde{B}^{n_j,k}_s+\kappa \int_0^t \<\tilde{v}_{n_j}(s), \Delta \phi \> \, ds.
			\end{split}
		\end{equation}
		Next, we need to establish the convergence for each term. Before doing so, we state a lemma that provides the regularity of $(\tilde{u},\tilde{v})$. Since the proof is similar to that of \cite[Corollary 3.7]{FL21 PTRF}, we omit it here.
		\begin{lemma}\label{lemma-regularity of limit process}
			The limit process $(\tilde{u},\tilde{v})$ satisfies the following bounds $\tilde{\P}$-a.s.,
			$$\|\tilde{v}\|_{L^\infty([0,T];L^2)} \leq C_{ub}, \quad \|\tilde{u}\|_{L^\infty([0,T];H^1)} \leq  C_{ub}.$$
		\end{lemma}
		We now discuss the convergence for the nonlinear and martingale terms in \eqref{eq-SWE2 existence appro}; the linear terms can be handled directly via point (2) above. By triangle inequality, it holds
		$$\aligned
		&\quad \bigg|\int_{0}^{t}  \big\< f(\tilde{u}_{n_j}(s)),\Pi_{n_j} \phi \big\> \, ds-\int_{0}^{t} \big\<f(\tilde{u}(s)),\phi \big\> \, ds\bigg|\\
		&\leq  \int_{0}^{t} \big|\big\<f(\tilde{u}_{n_j}(s))-f(\tilde{u}(s)),\Pi_{n_j} \phi \big\> \big|\, ds+\int_{0}^{t} \big|\big\<f(\tilde{u}(s)),\Pi_{n_j} \phi-\phi \big\>\big| \, ds.
		\endaligned $$
		For the first part, noticing that $f$ is Lipschitz, Cauchy-Schwarz inequality yields
		$$\aligned
		\int_{0}^{t} \big|\big\<f(\tilde{u}_{n_j}(s))-f(\tilde{u}(s)),\Pi_{n_j} \phi \big\> \big|\, ds& \leq \int_{0}^{t} \big\|f(\tilde{u}_{n_j}(s))-f(\tilde{u}(s))\big\|_{L^2} \|\Pi_{n_j} \phi\|_{L^2} \, ds \\
		&\lesssim \|\phi\|_{L^2} \int_{0}^{t} \|\tilde{u}_{n_j}(s)-\tilde{u}(s)\|_{L^2} \, ds\\
		&\leq T  \|\phi\|_{L^2} \|\tilde{u}_{n_j}-\tilde{u}\|_{C([0,T];L^2)},
		\endaligned $$
		which tends to zero due to the fact that $\tilde{\P}$-a.s., $\tilde{u}_{n_j}$ converges to $\tilde{u}$ in $C([0,T];L^2)$ as $j\rightarrow \infty$. For the second part of nonlinear term, it holds
		$$\aligned
		\int_{0}^{t} \big|\big\<f(\tilde{u}(s)),\Pi_{n_j} \phi-\phi \big\>\big| \, ds &\leq \int_{0}^{t} \|f(\tilde{u}(s))\|_{L^2} \|\Pi_{n_j} \phi-\phi \|_{L^2} \, ds \\
		&\lesssim \|\Pi_{n_j} \phi-\phi \|_{L^2}  \int_{0}^{t} \big(1+\|\tilde{u}(s)\|_{L^2}^2\big)^\frac{1}{2} \, ds.
		\endaligned $$
		Since the integral is bounded, it also vanishes as $j\rightarrow \infty$. Summarizing the above arguments, we obtain the convergence for the nonlinear term.
		
		As for the stochastic term, keep in mind that $\{\tilde{v}_{n_j}\}_{j}  \subset L^\infty([0,T];L^2)$ and $\tilde{v} \in L^\infty([0,T];L^2)$, then one can follow the proof of \cite[Theorem 2.2]{FGL21 JEE} to show that
		$$	\tilde{\E} \bigg[\sup_{t\in [0,T]} \bigg|\sum_{k} \int_0^t \big\<\tilde{v}_{n_j}(s), \sigma_{k}\cdot \nabla \Pi_{n_j} \phi\big\> \, d\tilde{B}_s^{n_j,k}-\sum_{k} \int_0^t \big\<\tilde{v}(s), \sigma_{k}\cdot \nabla \phi\big\> \, d\tilde{B}_s^{k}\bigg|\bigg] \rightarrow 0, \quad j\rightarrow \infty.$$
		Thanks to the above discussions, now we let $j\rightarrow \infty$ in \eqref{eq-SWE2 existence appro} and obtain that, $\tilde{\P}$-a.s., $(\tilde{u},\tilde{v},\tilde{B})$ satisfies \eqref{eq-SWE2 sol def} for all $t\in [0,T]$, yielding the existence of weak solution to equation \eqref{eq-SWE2}.
	\end{proof}
	
	Now we briefly explain why we cannot establish the pathwise uniqueness for \eqref{eq-SWE2}.
	\begin{remark}\label{rmk-not unique}
		Suppose $(u^1,v^1)$ and $(u^2,v^2)$ are two solutions to \eqref{eq-SWE2} with the same initial value, defined on the same probability space and driven by the same Brownian motions. Let $(w,z):=(u^1-u^2,v^1-v^2)$. To prove pathwise uniqueness, it suffices to show $(w,z) \equiv 0$.
		
		When estimating $d\mathcal{N}(t):=d\big(\|w(t)\|_{L^2}^2+\|z(t)\|_{H^{-1}}^2\big)$, two terms arise:
		$$2\kappa \big\<(I-\Delta)^{-1} z, \Delta z\> \quad \text{and} \quad  \sum_k \big\|(I-\Delta)^{-\frac{1}{2}} ( \sigma_k \cdot \nabla z)\big\|_{L^2}^2.$$
		If their sum could be controlled by $C_\kappa  \|z\|_{H^{-1}}^2$, then one would obtain
		$$d\mathcal{N}(t) \leq C_\kappa \mathcal{N}(t)+2\sum_k \big\<(I-\Delta)^{-1}z, \sigma_k \cdot \nabla z  \big\> \, dB_t^k;$$
		integrating in time and taking expectations, Gr\"onwall's lemma would imply $\mathcal{N}(t)\equiv 0$. However, we are unable to establish this key estimate, and therefore pathwise uniqueness remains open.
	\end{remark}

	\section{Scaling limit for \eqref{eq-SWE}}\label{sec-scaling limit SWE1}
	In this section, we present the proof of Theorem \ref{thm-scaling limit SWE1}. We will consider equations \eqref{eq-scaling limit} and show that the sequence of solutions $\{u^N\}_N$ converges weakly to the unique solution $\bar{u}$ of a deterministic wave equation under Assumption \ref{asp-2}.
	
	Recalling we have proved in Theorem \ref{thm-SWE1 wellposedness} that, for any $N\geq 1$, equation \eqref{eq-scaling limit} admits a weak solution $(u^N,v^N)\in L^\infty([0,T];H^1\times L^2)$. Since $\partial_t u^N=v^N$, we can deduce $\{u^N\}_N \subset W^{1,\infty}([0,T];L^2)$ as well. Moreover, similar discussions as those in Lemma \ref{lemma-fractional sobolev estimate} imply $\{v^N\}_N \subset W^{\frac{1}{3},4}([0,T];H^{-\rho})$. Hence applying Theorem \ref{thm-compact embedding}, we can claim that $\{\mathcal{L}^N\}_N$, the joint law of $\{(u^N,v^N)\}_{N}$, is tight on $\mathcal{X}=C([0,T];H^{1-\gamma}\times H^{-\gamma})$.
	Now we can provide
	\begin{proof}[Proof of Theorem \ref{thm-scaling limit SWE1}] The proof is divided into three steps for clarity.
		
		\textbf{Step 1.} We first establish the uniqueness of the limit equation \eqref{eq-limit wave equation}. Assume that $\bar{u}_1, \bar{u}_2 \in L^\infty([0,T];H^1)$ are two solutions corresponding to the same initial data, and set $\bar{v}_i:= \partial_t \bar{u}_i$ for $i=1,2$.
		Denote the pair $(\bar{w},\bar{z}):=(\bar{u}_1-\bar{u}_2,\bar{v}_1-\bar{v}_2)$, which satisfies the first-order system
		$$	\left\{
		\aligned
		&\partial_t \bar{w}=\bar{z},  \\
		& \partial_t \bar{z}=\Delta \bar{w}+ f(\bar{u}_1)-f(\bar{u}_2).
		\endaligned
		\right.$$
		Similarly to the proof of Theorem \ref{thm-SWE2 wellposedness}, we define $\bar{\mathcal{N}}(t):=\|\bar{w}(t)\|_{L^2}^2+\|\bar{z}(t)\|_{H^{-1}}^2$ and estimate its decay with respect to time:
		$$\aligned
		\frac{1}{2}\frac{d\bar{\mathcal{N}}(t)}{dt}&=\<\bar{w},\bar{z}\> +\<(I-\Delta)^{-1} \bar{z}, \Delta \bar{w}+ f(\bar{u}_1)-f(\bar{u}_2)\>\\
		&=\<(I-\Delta)^{-1} \bar{z},\bar{w} \>+\<(I-\Delta)^{-1} \bar{z},f(\bar{u}_1)-f(\bar{u}_2)\>  \\
		&\leq \|\bar{z}\|_{H^{-2}} \big( \|\bar{w}\|_{L^2}+\|f(\bar{u}_1)-f(\bar{u}_2)\|_{L^2} \big).
		\endaligned $$
		Applying Young's inequality and the Lipschitz property of $f$, we obtain
		$$\frac{1}{2}\frac{d\bar{\mathcal{N}}(t)}{dt} \lesssim \|\bar{z}\|_{H^{-2}}^2+\|\bar{w}\|_{L^2}^2\leq C \bar{\mathcal{N}}(t).$$
		Recalling that $\bar{\mathcal{N}}(0)=0$, Gr\"onwall's lemma yields $\bar{\mathcal{N}}(t)\equiv 0$ for all $t\in [0,T]$, and the uniqueness of the solution to \eqref{eq-limit wave equation} follows.
		
		\textbf{Step 2.} We then show the weak convergence result. For each $N\geq 1$, let $\eta^N$ denote  the joint law of $(u^N,v^N,B)$. By the compactness arguments above, the family $\{\eta^N\}_{N\geq 1}$ is tight on $\mathcal{X} \times \mathcal{Y}$. Utilizing Prohorov theorem and Skorohod representation theorem, we can find a subsequence $\{N_j\}_j$ and a probability space $(\bar{\Omega},\bar{\mathcal{F}},\bar{\P})$ with stochastic processes $\{(\bar{u}^{N_j},\bar{v}^{N_j},\bar{B}^{N_j})\}_{j\geq 1}$ and $(\bar{u},\bar{v},\bar{B})$ on it, such that
		
		(1) $\bar{\P}$-a.s., $(\bar{u}^{N_j},\bar{v}^{N_j},\bar{B}^{N_j})$ converges to $(\bar{u},\bar{v},\bar{B})$ in $\mathcal{X} \times \mathcal{Y}$ as $j\rightarrow \infty$;
		
		(2) for every $j\geq 1$, $(\bar{u}^{N_j},\bar{v}^{N_j},\bar{B}^{N_j})$ has the same joint law as $(u^{N_j},v^{N_j},B)$.
		
		Hence, for any $\phi \in C^\infty(\T^d)$, it holds
		$$\aligned
		\<\bar{u}^{N_j}(t),\phi\>&=\< u_0, \phi\>+\int_{0}^{t} \<\bar{v}^{N_j}(s),\phi\> \, ds, \\
		\<\bar{v}^{N_j}(t),\phi\>&=\<v_0,\phi\>+\int_{0}^{t} \<\bar{u}^{N_j}(s), \Delta \phi \> \, ds+\int_{0}^{t} \big\<f(\bar{u}^{N_j}(s)),\phi\big\> \, ds-\sum_k \int_{0}^{t} \big\<\bar{u}^{N_j}(s),\sigma_k^{N_j} \cdot \nabla \phi \big\> \, dB^k_s.
		\endaligned $$
		If we pass to the limit $j\rightarrow \infty$ in the above identities and obtain
		\begin{equation}\label{eq-limit eq weak form}
			\begin{split}
				\<\bar{u}(t),\phi\>&=\< u_0, \phi\>+\int_{0}^{t} \<\bar{v}(s),\phi\> \, ds, \\
				\<\bar{v}(t),\phi\>&=\<v_0,\phi\>+\int_{0}^{t} \<\bar{u}(s), \Delta \phi \> \, ds+\int_{0}^{t} \big\<f(\bar{u}(s)),\phi\big\> \, ds,
			\end{split}
		\end{equation}
		we can claim that the limit $(\bar{u},\bar{v})$ mentioned in the point (1) above solves the deterministic equation \eqref{eq-limit wave equation}. Thus, it remains to verify the convergence of each term.
		
		In fact, the linear terms converge evidently since $(\bar{u}^{N_j},\bar{v}^{N_j})$ converges $\bar{\P}$-a.s. to $(\bar{u},\bar{v})$ in $\mathcal{X}$. And we can apply again this convergence result, together with the Lipschitz property of $f$, to obtain the convergence of nonlinear term. It remains to show that, under the Assumption \ref{asp-2}, the martingale term vanishes as $j\rightarrow \infty$ in the sense of mean square. By It\^o isometry, it holds
		$$\aligned
		\E \bigg(\sum_k \int_{0}^{t} \big\<\bar{u}^{N_j}(s),\sigma_k^{N_j} \cdot \nabla \phi \big\> \, d\bar{B}^{N_j,k}_s\bigg)^{\! 2}
		&=\E \int_{0}^{t} \big\<\bar{u}^{N_j} \nabla \phi,Q^{N_j} \ast (\bar{u}^{N_j} \nabla \phi) \big\> ds\\
		&\leq \E \int_{0}^{t} \|\bar{u}^{N_j} \nabla \phi\|_{L^2} \big\|Q^{N_j} \ast \bar{u}^{N_j} \nabla \phi \big\|_{L^2} \, ds\\
		&\leq \big\|Q^{N_j}\big\|_{L^1}  \|\nabla \phi\|_{L^\infty}^2 \, \E \int_{0}^{t} \|\bar{u}^{N_j}\|_{L^2}^2 \, ds,\\
		\endaligned $$
		where the last step follows from the Young's inequality for convolution. By the uniform bound of $\{\bar{u}^{N_j}\}_j$, the expectation is finite, and thus the above estimate vanishes by \eqref{eq-asp on theta}. 
		
		In summary, the weak formulation \eqref{eq-limit eq weak form} holds for any $\phi\in C^\infty$; combining two equations and rewriting them in a more compact way, we obtain \eqref{eq-limit wave equation}.
		
		\textbf{Step 3.} Finally, we establish the strong convergence result. Recall that $\{u^N\}_N$ are probabilistically strong solutions which can be defined on the same probability space and converge to a deterministic limit $\bar{u}$. In this setting, convergence in law is equivalent to convergence in probability. Since we have shown that the laws of $\{u^N \}_N$ converge to the Dirac measure $\delta_{\bar{u}_\cdot}$ on $C([0,T];H^{1-\gamma})$, it follows that for any $\varepsilon>0$,
		$$\P \big\{\|u^N-\bar{u}\|_{C([0,T];H^{1-\gamma})} > \varepsilon\big\} \rightarrow 0, \quad N \rightarrow \infty.$$
		We now derive the corresponding moment estimate:
		$$\aligned
		\E \, \|u^N-\bar{u}\|_{C([0,T];H^{1-\gamma})} ^2&=\E \Big[\|u^N-\bar{u}\|_{C([0,T];H^{1-\gamma})} ^2 \, \textbf{1}_{\{\|u^N-\bar{u}\|_{C([0,T];H^{1-\gamma})} \leq \varepsilon\}}\Big]\\
		&\quad+ \E \Big[\|u^N-\bar{u}\|_{C([0,T];H^{1-\gamma})} ^2 \, \textbf{1}_{\{\|u^N-\bar{u}\|_{C([0,T];H^{1-\gamma})} >\varepsilon\}}\Big]\\
		&\lesssim \varepsilon^2+ 2C_{ub}^2 \, \P \big\{\|u^N-\bar{u}\|_{C([0,T];H^{1-\gamma})}> \varepsilon\big\}.
		\endaligned $$
		Since $\varepsilon>0$ is arbitrary, the right-hand side vanishes as $N \rightarrow \infty$, which yields the desired strong convergence result.
	\end{proof}

	\begin{remark}
		In the proof of Theorem \ref{thm-scaling limit SWE1}, we cannot repeat the proof of Theorem \ref{thm-SWE1 wellposedness} to show that $\{(u^N,v^N)\}_{N}$ is a Cauchy sequence in $L^2\big(\Omega;C([0,T]; L^2\times H^{-1})\big)$, namely
		$$\lim_{ M>N\rightarrow \infty} \E \bigg[\sup_{t\leq T} \Big(\big\|u^M(t)-u^N(t)\big\|_{L^2}^2+ \big\|v^M(t)-v^N(t)\big\|_{H^{-1}}^2\Big)\bigg]=0,$$
		where $(u^M,v^M)$ and $(u^N,v^N)$ are the solutions to equation \eqref{eq-scaling limit} driven by different vector fields $\{\sigma_k^M\}_k$ and $\{\sigma_k^N\}_k$.
		Indeed, when we apply the It\^o formula to estimate $d\|v^M-v^N\|_{H^{-1}}^2$, a quadratic covariation term appears:
		$$\aligned
		&\quad \sum_{k} \big\|(I-\Delta)^{-\frac{1}{2}}  \big(\sigma_k^M \cdot \nabla u^M-\sigma_k^N \cdot \nabla u^N\big) \big\|_{L^2}^2 \\
		&\leq 2 \sum_{k} \big\|(I-\Delta)^{-\frac{1}{2}}  \big( (\sigma_k^M-\sigma_k^N) \cdot \nabla u^M\big) \big\|_{L^2}^2+2\sum_{k} \big\|(I-\Delta)^{-\frac{1}{2}}  \big(\sigma_k^N \cdot \nabla (u^M-u^N)\big) \big\|_{L^2}^2.
		\endaligned $$
		We cannot show that the first term vanishes without additional convergence properties on $\{\sigma_k^N\}_N$.
	\end{remark}

	\section{Scaling limit for \eqref{eq-SWE2}}\label{sec-scaling limit SWE2}
	
	In this section, we prove Theorem \ref{thm-scaling limit SWE2} by dividing the arguments into two parts.
	In Section \ref{subsec-convergence for SWE2}, we present a sketch of the proof for weak convergence. Then, under an additional assumption on the nonlinear term $f$, we derive a quantitative estimate in Section \ref{subsec-quantitative SWE2}.
	
	\subsection{Weak convergence} \label{subsec-convergence for SWE2}
	We first establish the uniqueness of the limit equation \eqref{eq-limit WE2}. Let $(\bar{u}_1,\bar{v}_1), (\bar{u}_2,\bar{v}_2) \in L^\infty([0,T];H^1 \times L^2)$ be two solutions with the same initial value and denote $(\bar{w},\bar{z}):=(\bar{u}_1-\bar{u}_2,\bar{v}_1-\bar{v}_2)$. Then the pair satisfies the equation with $(\bar{w}(0),\bar{z}(0))=(0,0)$:
	$$	\left\{
	\aligned
	&\partial_t \bar{w}=\bar{z},  \\
	& \partial_t \bar{z}=\Delta \bar{w}+ f(\bar{u}_1)-f(\bar{u}_2) +\kappa \Delta \bar{z}.
	\endaligned
	\right.$$
	We define the energy $\bar{\mathcal{N}}(t):=\|\bar{w}(t)\|_{L^2}^2+\|\bar{z}(t)\|_{H^{-1}}^2$ and estimate its decay with respect to time:
	$$\aligned 
	\frac{1}{2}\frac{d\bar{\mathcal{N}}(t)}{dt}&=\<\bar{w},\bar{z}\> +\<(I-\Delta)^{-1} \bar{z}, \Delta \bar{w}+ f(\bar{u}_1)-f(\bar{u}_2) +\kappa \Delta \bar{z}\>\\
	&=\<(I-\Delta)^{-1} \bar{z},\bar{w}+f(\bar{u}_1)-f(\bar{u}_2) \> +\kappa \<(I-\Delta)^{-1}\bar{z}, \Delta \bar{z} \>.
	\endaligned $$
	In fact, one can easily deduce that $\<(I-\Delta)^{-1}\bar{z}, \Delta \bar{z} \> \leq 0$, and therefore we utilize Lipschitz property of $f$ and Young's inequality to obtain
	$$\frac{1}{2}\frac{d\bar{\mathcal{N}}(t)}{dt} \leq \|(I-\Delta)^{-1} \bar{z}\|_{L^{2}}  \big(\|\bar{w}\|_{L^2}+\|f(\bar{u}_1)-f(\bar{u}_2) \|_{L^2} \big) \lesssim \|\bar{z}\|_{H^{-2}}^2+\|\bar{w}\|_{L^2}^2 \leq C \bar{\mathcal{N}}(t).$$
	Then Gr\"onwall's lemma yields $\bar{\mathcal{N}}(t)\equiv 0$ for all $t\in [0,T]$, and the uniqueness follows.
	
	Next, by arguments similar to those in the proof of Theorem \ref{thm-scaling limit SWE1}, we can construct a probability space $(\bar{\Omega},\bar{\mathcal{F}},\bar{\P})$ with stochastic processes $\{(\bar{u}^{N_j},\bar{v}^{N_j},\bar{B}^{N_j})\}_{j\geq 1}$ and $(\bar{u},\bar{v},\bar{B})$ on it, such that  for any $j\geq 1$, $(\bar{u}^{N_j},\bar{v}^{N_j},\bar{B}^{N_j})$ has the same joint law as $(u^{N_j},v^{N_j},B)$; and $\bar{\P}$-a.s., the former triple converges to $(\bar{u},\bar{v},\bar{B})$ in $\mathcal{X} \times \mathcal{Y}$ as $j\rightarrow \infty$.
	In this setting, the weak convergence of the linear and nonlinear terms follows directly from the convergence of  $(\bar{u}^{N_j},\bar{v}^{N_j})$ to $(\bar{u},\bar{v})$ in $\mathcal{X}=C([0,T];H^{1-\gamma} \times H^{-\gamma})$, together with the Lipschitz property of $f$.
	
	Finally, using the uniform bound  \eqref{eq-asp on v ub}, one can adapt discussions similar to Step 2 in the proof of Theorem \ref{thm-scaling limit SWE1} to show that the martingale term vanishes in $L^2(\bar{\Omega})$. We omit the details for brevity.

	\subsection{Quantitative estimate} \label{subsec-quantitative SWE2}
	Here we present a quantitative estimate for the difference of $\{u^N\}_N$ and $\bar{u}$ in the space $C([0,T];H^{-a})$ for $a\in (0,\frac{1}{2})$.
	To begin with, we give a key lemma for later use.
	\begin{lemma}\label{lemma-Cb2 f}
		Suppose $f\in C_b^2(\R)$, $u^N$ and $\bar{u}$ the solutions to \eqref{eq-scaling limit SWE2}, \eqref{eq-limit WE2} respectively, then for $a\in (0,\frac{1}{2})$ and $d=2,3$, it holds
		\begin{equation}\label{eq-f sobolev estimate}
			\big\|f(u^N_s)-f(\bar{u}_s)\big\|_{H^{-a-\frac{1}{2}}(\T^d)} \lesssim \|u^N_s-\bar{u}_s\|_{H^{-a}(\T^d)}.
		\end{equation}
	\end{lemma}
	
	\begin{proof}
		We apply mean value theorem to write
		$$f(u^N_s)-f(\bar{u}_s)=(u^N_s-\bar{u}_s) \int_{0}^{1} f'\big(\bar{u}_s+\theta(u^N_s-\bar{u}_s)\big) \, d\theta=:(u^N_s-\bar{u}_s) \mathcal{M}_s(x,\omega).$$
		Recall the Sobolev product inequality: if $s_1,s_2 <\frac{d}{2}$ and $s_1+s_2>0$, then
		$$	\|gh\|_{H^{s_1+s_2-\frac{d}{2}}(\T^d)} \lesssim \|g\|_{H^{s_1}(\T^d)} \|h\|_{H^{s_2}(\T^d)}. $$
		For $d=2,3$ and  $a\in(0,\frac12)$, we choose $\lambda=\frac{d-1}{2}$, so that $\lambda< \frac{d}{2}$ and $\lambda-a>0$. It follows that
		$$\big\|f(u^N_s)-f(\bar{u}_s) \big\|_{H^{-a-\frac{1}{2}}(\T^d)} \lesssim \|u^N_s-\bar{u}_s\|_{H^{-a}(\T^d)} \|\mathcal{M}_s\|_{H^\lambda(\T^d)}.$$
		Since $\lambda\leq 1$ for $d=2,3$, we may further bound $\|\mathcal{M}_s\|_{H^\lambda}$ by $\|\mathcal{M}_s\|_{H^1}$, which is controlled by
		$$\aligned
		&\quad \|\mathcal{M}_s\|_{H^1}^2=\|\mathcal{M}_s\|_{L^2}^2+\|\nabla \mathcal{M}_s\|_{L^2}^2\\
		&=\int_{\T^d} \Big|\int_{0}^{1} f'\big(\bar{u}_s+\theta(u^N_s-\bar{u}_s)\big) \, d\theta\Big|^2 \, dx+\int_{\T^d} \Big|\int_{0}^{1} \nabla \Big( f'\big(\bar{u}_s+\theta(u^N_s-\bar{u}_s)\big) \Big)\, d\theta\Big|^2 \, dx\\
		&\leq C \|f'\|_{L^\infty}^2+\int_{\T^d} \Big|\int_{0}^{1}  f''\big(\bar{u}_s+\theta(u^N_s-\bar{u}_s)\big) \big((1-\theta) \nabla \bar{u}_s+\theta \nabla u^N_s\big)\, d\theta\Big|^2 \, dx.
		\endaligned $$
		We denote the last integral term as $J(s)$. By triangle inequality,
		$$\aligned
		&\quad \Big|\int_{0}^{1}  f''\big(\bar{u}_s+\theta(u^N_s-\bar{u}_s)\big) \big((1-\theta) \nabla \bar{u}_s+\theta \nabla u^N_s\big)\, d\theta\Big| \\
		&\leq \int_{0}^{1} \Big|f''\big(\bar{u}_s+\theta(u^N_s-\bar{u}_s)\big)\Big|  \Big((1-\theta) \big|\nabla \bar{u}_s\big|+\theta \big|\nabla u^N_s\big|\Big) \, d\theta\\
		&\leq \frac{1}{2} \|f''\|_{L^\infty} \Big( \big|\nabla \bar{u}_s\big|+ \big|\nabla u^N_s\big| \Big),
		\endaligned $$
		and therefore we can take integral on the torus $\T^d$ to get
		$$J(s) \leq \frac{1}{4} \|f''\|_{L^\infty}^2 \int_{\T^d}  \Big( \big|\nabla \bar{u}_s\big|+ \big|\nabla u^N_s\big| \Big)^2 \, dx \leq \frac{1}{2} \|f''\|_{L^\infty}^2  \Big(\|\nabla \bar{u}_s\|_{L^2}^2+\|\nabla u^N_s\|_{L^2}^2\Big).$$
		Now we can utilize the facts that $\bar{u}$ is bounded in $L^\infty([0,T];H^1)$ and $\{u^N\}_N$ has a uniform upper bound which is independent of $\omega$ in this space to claim that $\|\mathcal{M}_s\|_{H^1}^2$ can be controlled by some finite constant, which implies that \eqref{eq-f sobolev estimate} holds.
	\end{proof}
	
	\begin{remark}\label{rmk-dimension restrction}
		In \eqref{eq-f sobolev estimate}, one may also consider other Sobolev norms on the left-hand side, such as $H^{-a-1}$ and $H^{-a-\epsilon}$. However, applying the Sobolev product inequality requires the condition $\lambda-a>0$. Moreover, in order to
		bound $\|\mathcal{M}_s\|_{H^\lambda}$ by $\|\mathcal{M}_s\|_{H^1}$ and to utilize the assumption $f\in C_b^2(\R)$, it is necessary to let $\lambda \leq 1$.
		
		These two requirements together impose restrictions on the admissible spatial dimensions. Since the cases $d=2,3$ are the most natural and physically relevant for the wave equation, we therefore choose the $H^{-a-\frac{1}{2}}$-norm in the above lemma and concern about the results for $d=2,3$.
	\end{remark}

	Based on the above lemma, now we can establish the quantitative convergence result of Theorem \ref{thm-scaling limit SWE2}. Define $P_t=e^{\kappa t \Delta}$, then the second equation in \eqref{eq-scaling limit SWE2}  can be written as
	$$v^N_t=P_t v_0+\int_{0}^{t} P_{t-s} \Delta u^N_s \, ds+\int_{0}^{t} P_{t-s} f(u_s^N) \, ds+\sum_{k} \int_{0}^{t} P_{t-s} \big(\sigma_{k}^N \cdot \nabla v^N_s\big) \, dB^k_s. $$
	We denote the stochastic convolution as $Z_t^N$.
	Similarly, we write the second equation of \eqref{eq-limit WE2} as
	$$\bar{v}_t=P_t v_0+\int_{0}^{t} P_{t-s} \Delta \bar{u}_s \, ds+\int_{0}^{t} P_{t-s} f(\bar{u}_s) \, ds.$$
	If we make difference and take $H^{-a}$-norm, then we obtain
	$$	\|v_t^N-\bar{v}_t\|_{H^{-a}}^2 \lesssim \bigg\|\int_{0}^{t} P_{t-s} \Delta (u_s^N-\bar{u}_s) \, ds\bigg\|_{H^{-a}}^2+\bigg\|\int_{0}^{t} P_{t-s} \big(f(u_s^N)-f(\bar{u}_s) \big) \, ds\bigg\|_{H^{-a}}^2+\big\|Z_t^N \big\|_{H^{-a}}^2.$$
	Taking integral with respect to time from $0$ to $T$ yields
	$$\int_{0}^{T} \|v_t^N-\bar{v}_t\|_{H^{-a}}^2 \, dt \lesssim I_1(T)+I_2(T)+ \int_{0}^{T} \big\|Z_t^N\big\|_{H^{-a}}^2 \, dt,$$
	where $I_1(T) $ and $I_2(T)$ are defined in an obvious way.
	\iffalse
	$$\aligned
	I_1(T)&:=\int_{0}^{T}\bigg\|\int_{0}^{t} P_{t-s} \Delta (u_s^N-\bar{u}_s) \, ds\bigg\|_{H^{-a}}^2 \, dt,\\
	I_2(T)&:=\int_0^T \bigg\|\int_{0}^{t} P_{t-s} \big(f(u_s^N)-f(\bar{u}_s) \big) \, ds\bigg\|_{H^{-a}}^2\, dt.
	\endaligned $$
	\fi
	We will estimate each term one by one. By Lemma \ref{lemma-semigroup property}, it holds
	$$\aligned
	I_1(T) \leq \frac{1}{\kappa^2} \int_{0}^{T} \big\|\Delta(u_s^N-\bar{u}_s)\big\|_{H^{-a-2}}^2 \, ds \leq \frac{1}{\kappa^2} \int_{0}^{T} \big\|u_s^N-\bar{u}_s\big\|_{H^{-a}}^2 \, ds.
	\endaligned $$
	Since $u^N_s-\bar{u}_s=\int_{0}^{s} (v^N_r-\bar{v}_r) \, dr$, we can deduce that for any $s\in [0,T]$,
	\begin{equation}\label{eq-relation u v}
		\|u^N_s-\bar{u}_s\|_{H^{-a}}^2\leq \bigg(\int_{0}^{s} \|v^N_r-\bar{v}_r\|_{H^{-a}} \, dr\bigg)^2 \leq s \int_{0}^{s} \|v_r^N-\bar{v}_r\|_{H^{-a}}^2 \, dr,
	\end{equation}
	and therefore
	$$I_1(T) \leq \frac{T}{\kappa^2} \int_{0}^{T}  \int_{0}^{s} \|v_r^N-\bar{v}_r\|_{H^{-a}}^2 \, dr ds.$$
	
	Then we turn to the term $I_2(T)$. Applying Lemmas \ref{lemma-semigroup property} and \ref{lemma-Cb2 f}, we have
	\begin{equation}\label{eq-f sobolev}
		I_2(T) \leq \frac{1}{\kappa^2} \int_{0}^{T} \big\|f(u_s^N)-f(\bar{u}_s)\big\|_{H^{-a-2}}^2 \, ds \lesssim \frac{1}{\kappa^2} \int_{0}^{T} \|u_s^N-\bar{u}_s\|_{H^{-a}}^2 \, ds.
	\end{equation}
	Substituting \eqref{eq-relation u v} into the above one, we obtain the estimate for this term:
	$$I_2(T) \lesssim \frac{T}{\kappa^2} \int_{0}^{T} \int_{0}^{s} \|v_r^N-\bar{v}_r\|_{H^{-a}}^2 \, drds.$$
	Finally, we discuss the stochastic convolution. By Corollary \ref{coro-stochastic convolution}, it holds
	$$\E \int_{0}^{T} \|Z_t^N\|_{H^{-a}}^2 \, dt \lesssim_{\epsilon,T,\kappa} C_{ub}^2 \big\|Q^N\big\|_{L^1}^\frac{2(a-\epsilon)}{d}.$$
	Now we can summarize the arguments and arrive at
	$$\E \bigg[\int_{0}^{T} \|v_t^N-\bar{v}_t\|_{H^{-a}}^2 \, dt\bigg] \lesssim \frac{T}{\kappa^2} \, \E \bigg[\int_{0}^{T}  \int_{0}^{s} \|v_r^N-\bar{v}_r\|_{H^{-a}}^2 \, dr ds\bigg]+C_{ub}^2 \big\|Q^N\big\|_{L^1}^\frac{2(a-\epsilon)}{d},$$
	and Gr\"onwall's lemma leads to
	$$\E \bigg[\int_{0}^{T} \|v_t^N-\bar{v}_t\|_{H^{-a}}^2 \, dt \bigg]\lesssim_{T,\kappa} C_{ub}^2 \big\|Q^N\big\|_{L^1}^\frac{2(a-\epsilon)}{d}.$$
	Applying \eqref{eq-relation u v} again, we can take expectation and  obtain
	$$	\E \, \bigg[\sup_{t\in [0,T]} \|u_t^N - \bar{u}_t\|_{H^{-a}}^2 \bigg]	\leq T \, \E \bigg[\int_0^T \|v_s^N - \bar{v}_s\|_{H^{-a}}^2 \, ds\bigg] \lesssim_{T,\kappa}  C_{ub}^2 \big\|Q^N\big\|_{L^1}^\frac{2(a-\epsilon)}{d}.$$
	The proof of Theorem \ref{thm-scaling limit SWE2} is complete.

	\bigskip
	\noindent\textbf{Acknowledgements.} The second author is supported by the National Key R\&D Program of China (No.  2024YFA1012301).

\end{document}